\documentclass[final]{siamart0216}

\usepackage{lipsum}
\usepackage{amsfonts, scalefnt}
\usepackage{graphicx}
\usepackage{epstopdf}
\usepackage{algorithm}
\usepackage{algpseudocode} 
\usepackage{braket}

\usepackage{psfrag}
\usepackage{amssymb,bbm,latexsym}
\usepackage{colortbl,fancyhdr}
\usepackage{ caption, subcaption}
\usepackage{float,booktabs,rotating}

\usepackage[percent]{overpic}
\usepackage{tikz}
\usetikzlibrary{shapes,arrows}

\ifpdf
  \DeclareGraphicsExtensions{.eps,.pdf,.png,.jpg}
\else
  \DeclareGraphicsExtensions{.eps}
\fi

\newcommand{\Real}{\mathbb{R}}

\newcommand{\N}{\mathbb{N}}
\newcommand{\mc}{\mathcal}
\newcommand{\bx}{\mathbf{x}}
\newcommand{\bu}{\mathbf{u}}

\newcommand{\ba}{\mathbf{a}}
\newcommand{\bb}{\mathbf{b}}
\newcommand{\by}{\mathbf{y}}
\newcommand{\ta}{_\textup{a}}
\newcommand{\tb}{_\textup{b}}
\newcommand{\tc}{_\textup{c}}
\newcommand{\dz}{d_\textup{z}}
\newcommand{\dy}{d_\textup{y}}

\newcommand{\dv}{d_\textup{v}}
\newcommand{\dm}{d_\mu}
\newcommand{\dl}{d_\lambda}
\newcommand{\dn}{d_\nu}
\newcommand{\dxi}{d_\xi}
\newcommand{\ds}{d_\text{s}}
\newcommand{\dsj}{d_{\text{s},j}}
\newcommand{\al}[1]{\mathcal{L}_\text{aug} \left( #1, \rho \right)}
\newcommand{\x}[1]{x^{(#1)}}
\newcommand{\QP}{_\textup{QP}}
\newcommand{\G}{_\textup{G}}
\newcommand{\Gj}{_{\textup{G},j}}
\newcommand{\NL}{_\textup{NL}}
\newcommand{\g}{\mathbbm{g}}
\newcommand{\h}{\mathbbm{h}}
\newcommand{\J}{\mathbbm{J}}
\newcommand{\ii}{^{(i)}}
\newcommand{\kk}{^{(k)}}
\newcommand{\mm}{^{(m)}}
\newcommand{\ir}{^{(i_l)}}

\newcommand{\ip}{^{(i+1)}}
\newcommand{\im}{^{(i-1)}}

\newcommand{\Lagr}{\mathcal{L}}

\newcommand{\Du}{\mathbf{d}_\text{u}}

\newcommand{\Dya}{\mathbf{d}_\text{y,a}}
\newcommand{\Dyb}{\mathbf{d}_\text{y,b}}
\newcommand{\Dyc}{d_\text{y,c}}
\newcommand{\bs}{\boldsymbol}

\newcommand{\Proja}[2]{\Pi_{#1} \left(#2\right)}

\newcommand{\txx}{\text{x}}
\newcommand{\txu}{\text{u}}
\newcommand{\txa}{\text{a}}

\newcommand{\C}{\mathcal{C}}

\newcommand{\Q}{\mathcal{Q}}

\newcommand{\R}{\mathbb{R}}

\newcommand{\norm}[1]{\left\| #1 \right\|_2}

\newtheorem{rem}{{\sc Remark}}
\crefname{rem}{remark}{Remark}
\numberwithin{rem}{section}

\newtheorem{lem}{{\sc Lemma}}
\crefname{lem}{lemma}{Lemma}
\numberwithin{lem}{section}
\newtheorem{assumption}{{\sc Assumption}}
\crefname{assumption}{assumption}{Assumption}
\newtheorem{standing}[assumption]{{\sc Standing Assumption}}
\crefname{standing}{Standing Assumption}{Standing Assumption}

\numberwithin{equation}{section}
\numberwithin{theorem}{section}
\numberwithin{assumption}{section}
\numberwithin{figure}{section}
\numberwithin{table}{section}

\newcommand{\TheTitle}{A Projected Gradient and Constraint Linearization Method for Nonlinear Model Predictive Control}
\newcommand{\TheTitleShort}{Projected Gradient \& Constraint Linearization Method for NMPC}
\newcommand{\TheAuthors}{G. Torrisi, S. Grammatico, R. S. Smith and M. Morari}

\headers{\TheTitleShort}{\TheAuthors}

\title{{\TheTitle}}

\author{
  Giampaolo Torrisi\thanks{Automatic Control Laboratory,
        ETH Zurich,  Switzerland. Emails: {\tt\small \{torrisig, rsmith, morari\}@control.ee.ethz.ch}.}
 \and Sergio Grammatico\thanks{Control Systems group, Department of Electrical Engineering, Eindhoven University of Technology, The Netherlands. Email: {\tt\small s.grammatico@tue.nl}. } \and Roy S. Smith\footnotemark[1] \and Manfred Morari\footnotemark[1]
}

\usepackage{amsopn}
\DeclareMathOperator{\diag}{diag}

\ifpdf
\hypersetup{
  pdftitle={\TheTitle},
  pdfauthor={\TheAuthors}
}
\fi

\begin{document}

\maketitle

\begin{abstract}
Projected Gradient Descent denotes a class of iterative methods for solving optimization programs. Its applicability to convex optimization programs has gained significant popularity for its intuitive implementation that involves only simple algebraic operations. In fact, if the projection onto the feasible set is easy to compute, then the method has low complexity.
On the other hand, when the problem is nonconvex, e.g. because of nonlinear equality constraints, the projection becomes hard and thus impractical. In this paper, we propose a projected gradient method for Nonlinear Programs (NLPs) that only requires projections onto the linearization of the nonlinear constraints around the current iterate, similarly to Sequential Quadratic Programming (SQP). Although the projection is easier to compute, it makes the intermediate steps unfeasible for the original problem.
As a result, the gradient method does not fall either into the projected gradient descent approaches, because the projection is not performed onto the original nonlinear manifold, or into the standard SQP, since second-order information is not used. 
For nonlinear smooth optimization problems, we analyze the similarities of the proposed method with SQP and assess its local and global convergence to a Karush-Kuhn-Tucker (KKT) point of the original problem.
Further, we show that nonlinear Model Predictive Control (MPC) is a promising application of the proposed method, due to the sparsity of the resulting optimization problem. We illustrate the computational efficiency of the proposed method in a numerical example with box constraints on the control input and a quadratic terminal constraint on the state variable. \footnote{Preprint submitted to SIAM Journal on Control and Optimization.} \footnote{The work presented in this manuscript has similarities with \cite{Alg_preliminary_CDC2016}. The technical differences are explained in Section 1.} 
\end{abstract}



\section{Introduction}

The projected gradient method is an established approach for solving convex optimization problems. The subject has been extensively investigated over the last decades, developing algorithms that guarantee best performance for convex and strongly convex problems, see \cite{nesterov2013introductory, bertsekas1999nonlinear}. Recently, the Nesterov's accelerated gradient method has been applied to linear Model Predictive Control (MPC), and a priori worst-case bounds for finding a solution with prespecified accuracy has been derived
\cite{nesterov1983method, richterFastMPC}. 
When the optimization problem is a general nonlinear program, the gradient method can still be used for finding a KKT point \cite{rosen1961gradient2}. In particular, for general nonconvex constraints in the form of a nonlinear manifold, the projection onto the feasible set is performed in two stages. First, the projection is derived onto the tangent space to the nonlinear manifold, which in general is a polyhedron. Then, the determined point is projected again onto the original nonlinear manifold, via some strategy guaranteed to determine a feasible point that improves the objective function. While ensuring convergence, this second projection is in general computationally expensive, hence the method is not recommended in practice for solving nonlinear MPC problems. 

In this paper, we analyze a gradient method for Nonlinear Programs (NLPs) that only requires projections onto the tangent space, obtained by linearization of the nonlinear manifold around the current iterate. Note that this determines a sequence of points that are not necessarily feasible for the original NLP. 
Thus, standard projected gradient method results do not apply to prove convergence. 

Linearized constraints are instead considered in Sequential Quadratic Programming (SQP), which is an established method to determine a local solution to a smooth nonconvex NLP. The solution is determined via a sequence of iterates, each obtained as the solution to a Quadratic Program (QP), that are usually called the major iterations. 
In turn, each QP is solved via so-called minor iterations using available convex optimization methods \cite{nocedal1999numerical, Robinson, Boggs_survey}. Typically, each QP has as objective function a second-order approximation of the Lagrangian function of the nonlinear problem and as constraints the linearization of the nonlinear manifold, both computed at the current iterate. The solution of the QP updates the current iterate and then the next QP is formulated. 

The basic SQP method is equivalent to the Newton's method applied to the KKT conditions of the original nonlinear optimization problem, thus it is locally quadratically convergent \cite{tapia1974stable}. Since the required Hessian of the Lagrangian is expensive to compute and it is not guaranteed to be positive definite on every subspace far from the solution,  a suitable approximation of the Hessian of the Lagrangian is typically used, e.g. quasi-Newton or Broyden-Fletcher-Goldfarb-Shanno (BFGS) updates, that guarantee local superlinear convergence
\cite{broyden1973local, boggs1982linear, coleman1990characterizations, nocedal1980updating}. 
Global convergence is usually obtained via a line-search approach. Merit functions are considered that comprise both the objective and the constraint functions, e.g. in the form of an augmented Lagrangian \cite{fletcher1972, gill_SQP_theory, powell1986recursive}. Then by appropriate tuning of some penalty parameters, the solution to the QP is proven to be a descent direction for the merit function. A line-search then determines a step size for the convergence of the method. 
Several contributions in the literature have discussed different reformulations that trade off theoretical convergence guarantees and computational complexity, see the review in \Cref{sec:Iterative_Methods} \cite{Boggs_survey, powell1986recursive, han1976superlinearly, powell1978fast, gill_SQP_theory}. 
Indeed, commercial numerical solvers use this technique for approaching a KKT point of an NLP \cite{gill2005snopt}. 
 An alternative approach for establishing global convergence is the trust-region method, where additional constraints are included in the optimization program \cite{Boggs_survey, dennis1996numerical, Tenny2004}.

The proposed gradient algorithm in this paper can be seen as an incomplete SQP where, instead of solving each generated QP, only one  gradient step is computed for the QP and then projected onto the linearized constraint. 
Some literature has proposed solving the QP program inexactly, e.g. by bounding the suboptimality of the estimated solution to the QP to recover some rate of convergence for the SQP \cite{dembo1982inexact, leibfritz1999inexact, murray1995sequential}. However, our approach does not fall into this class of methods, 
as only one gradient step  for each QP is in general not enough for reaching the desired level of suboptimality. Moreover, for these approaches second-order information on the Lagrangian is necessary, which in contrast is not required for our proposed algorithm. 

An approach that is widely considered in the literature is the Real-Time Iteration (RTI), specifically oriented to MPC applications \cite{Diehl2002577}. 
 It   basically yields an approximation to the NLP solution based on  the SQP method. Instead of iterating the solution to the QPs until a KKT point is encountered, only one QP is solved. Further technical differences concern the computation of the Jacobian matrices -- the QPs are formed using the Jacobian matrices from the previous time step -- and how the initial state is embedded into the optimization problem. These refinements allow the next control input to be rapidly calculated before the linearization is found for the upcoming time step.  
  Suboptimality and closed-loop considerations of RTI are discussed in \cite{Diehl2005b}.


In this paper we show local and global convergence properties of the proposed gradient algorithm, leveraging established SQP results in the literature. Local conditions are derived when each gradient step is directly employed to update the current iterate. As in standard gradient method,  ensuring convergence of the algorithm requires that some conditions have to be set on the gradient step size, typically depending on second-order information of the considered problem -- in our case, the Lipschitz constant of the Lagrangian function. Under a particular assumption on the Hessian, the algorithm converges with linear rate, as expected for first-order methods for NLPs \cite{ghadimi2016accelerated}. This is guaranteed to work close to a local optimum only. For the practical use of the algorithm, global convergence is required instead. Since updating the current iterate with this gradient step might not guarantee convergence, a variable step size is considered. Analogously to SQP, a merit function in the form of an augmented Lagrangian function weights the optimality and unfeasibility of the iterates and is employed in the line search for determining the step size.

Finally, we notice that sparsity considerably reduces  the computational complexity of the problem. In particular, we show that nonlinear MPC is particularly well suited for applying the proposed method, due to the structure of the constraints generated by causal model dynamics. Similarly to the gradient method for linear MPC, easy-to-project constraints can be efficiently included in the nonlinear MPC formulation \cite{richterFastMPC}. 
Furthermore, in the presence of a quadratic terminal cost and constraints that ensure closed-loop stability \cite{morari1999model, Allgower_terminal, Kothare2000}, we show that the projection can be computed in closed form, making the proposed algorithm computationally efficient.

Preliminary results have been accepted for publication in \cite{Alg_preliminary_CDC2016}, where the global convergence of a similar algorithm employing only primal iterates is proven. In this work, instead, the combined use of primal and dual variable iterates and other technical improvements in the considered augmented Lagrangian function reduce considerably the resulting computational complexity. In \cite{Alg_preliminary_CDC2016} the effect of some heuristics is also analyzed, that yield an interesting speed-up in the computational time specifically for MPC problems. 

 


The remainder of this paper is organized as follows. Given the similarity of the proposed method to SQP, the standard SQP method is reviewed in Section $2$. Then, the  proposed algorithm is presented in Section $3$. In Section $4$ we show the convergence of the algorithm, and then in Section $5$ we discuss the practical implementation for general problems and nonlinear MPC. Section $6$ shows numerical experiments on a benchmark example.

\section{Iterative Methods for Nonlinear Optimization Problems}
\label{sec:Iterative_Methods}

We consider the equality constrained nonlinear optimization problem (NLP)
\begin{equation} \label{eq:original_problem}
\begin{split}
\min_{z \in \Real^n} \hspace{0.3cm} &  J(z)  \\
\text{s.t.} \hspace{0.35cm} 
&g(z) \leq 0 \\
& h(z) = 0, \\
\end{split}
\end{equation}
%
where the functions $J:\Real^n~\rightarrow~\Real$
, $g:\Real^n~\rightarrow~\Real^m$ 
and $h:\Real^n~\rightarrow~\Real^p$ are twice continuously differentiable functions, possibly nonconvex. 

Let us define the Lagrangian function of the NLP in \eqref{eq:original_problem} as
\begin{equation*}
\Lagr (z, \lambda, \nu) :=  J (z) +  g(z)^\top \lambda +  h (z)^\top \nu,
\end{equation*}
with Lagrange multiplier vector $\lambda\in \Real^m_{\geq 0}$
and $\nu \in \Real^p$.
 
We call $z^{\star}$ a critical point of \eqref{eq:original_problem} if it satisfies the first order conditions with strict complementarity \cite{Boyd2004_ConvexOpt}, i.e., there exist $\lambda^\star \in \Real^m_{\geq 0}$ and $\nu^\star \in \Real^p$ such that
\begin{equation}
\label{eq:FOC}
\begin{split} 
&\nabla \Lagr (z^\star,\lambda^\star,\nu^\star) = \nabla J (z^\star) + \nabla g (z^\star) \lambda^\star + \nabla h (z^\star) \nu^\star = 0,\\
&\diag(\lambda^\star) g(z^\star) = 0 \\
&\lambda^\star_j > 0 \text{ if } g_j (z^\star) = 0 \\
&h(z^\star) = 0. 
\end{split}
 \end{equation}
Let us assume that the NLP in \eqref{eq:original_problem} has a finite number of critical points. 
 
%

Iterative methods generate a sequence $\left( z\ii \right)_{i=1}^{\infty}$ (major iterations) to determine a critical point $z^\star$. In \Cref{sec:SQP_Method}, the state-of-the-art Sequential Quadratic Programming is reviewed, following the approach in \cite{gill_SQP_theory}. To generate each $z\ii$, a sequence of minor iterations is required, and these are the intermediate steps that involve solving a Quadratic Program. 
 In \Cref{sec:Prp_Method}, we propose an alternative method to determine a critical point $z^\star$ by  computing only major iterations of the problem. At the $i$-th major iteration, $z\ip$ is directly derived via a projected gradient step onto a linearization of the constraint around the current iterate $z\ii$.

\subsection{Sequential Quadratic Programming}
\label{sec:SQP_Method}

The Sequential Quadratic Programming (SQP) method updates the sequence $\left( z\ii \right)_{i}$ via the solution of a sequence of Quadratic Programs (QPs). In particular, given the current iterate $z\ii$, $i \in \mathbb{N}$,
the method generates the QP
 \begin{equation} \label{eq:QP_z}
 \begin{split}
 \dz\ii := \arg \min_{ \dz } \hspace{0.3cm} & \textstyle \frac12
   {\dz}^\top H\ii \dz +  \nabla J (z\ii)^\top \dz   \\  
   \text{s.t. } 
   \hspace{0.25cm} &g(z\ii) + \nabla g(z\ii)^\top \dz \leq 0  \\
 &h(z\ii) + \nabla h(z\ii)^\top \dz = 0,  \\
 \end{split}
 \end{equation}
 where $ H\ii$ is either the exact Hessian of the Lagrangian $\mathcal{L}$ of the NLP in \eqref{eq:original_problem}, or an appropriate approximation.
 The dual variables  $\lambda\QP\ii \in \R^m_{\geq 0}$ and $\nu\QP\ii \in \R^p$ are associated with the inequality and equality constraints, respectively. The resulting KKT conditions for the QP in \eqref{eq:QP_z} are:
 \begin{equation}
 \label{eq:KKT_QP_z}
 \begin{split}
& H\ii \dz\ii + \nabla J(z\ii) + \nabla g (z\ii) \ \lambda\QP\ii +  \nabla h(z\ii) \ \nu\QP\ii = 0
\\ &  \diag(\lambda\QP\ii) \left(g (z\ii) + \nabla g (z\ii)^\top \dz\ii \right) = 0 \\
&h(z\ii) + \nabla h(z\ii)^\top \dz\ii = 0.
 \end{split}
 \end{equation}
 Based on the solution $\dz\ii$ to \eqref{eq:QP_z}, the sequence is updated as
\begin{equation} \label{eq:update}
z\ip := z\ii + t\ii \dz\ii,
\end{equation} 
where $t\ii \in (0,1]$ is a step size to be determined.

To prove convergence results for the SQP methods, some regularity and boundedness assumptions  are typically considered \cite[Assumptions (i)-(iii)]{gill_SQP_theory}. Some of these assumptions will additionally hold throughout the paper and they will be denoted as standing assumptions. 

Consider a generic optimization problem, indicated in the form of \eqref{eq:original_problem}, a feasible solution $z$ and the set of the active inequality constraints as 
\begin{equation*}
\mc{I}\NL (z) := \Set{j\in \{ 1, \ldots, m \} | g_j(z) = 0}.
\end{equation*}
Then, the vector $z$ is said to be regular if the equality constraint gradients $\nabla h_i (z)$, for all $i \in 1, \ldots, p$, and the active inequality constraints $\nabla g_j (z)$, for all $j \in \mc{I}\NL (z)$,  are linearly independent.

 \smallskip
 \begin{assumption} 
 \label{ass:1}
The matrices $ \{ H\ii \}_{i=1}^{\infty}$ are positive definite, with bounded condition number, and smallest eigenvalue uniformly bounded away from zero, i.e., $\exists \gamma > 0$ such that, for all $i \in \N$,
$
d^\top H\ii d \geq \gamma \norm{d}^2 \ \ \forall d\in \Real^n.
$
\end{assumption}
 \smallskip
\begin{standing}
 \label{ass:2}
For all $i \in \N$, the QP in \eqref{eq:QP_z} is feasible.
\end{standing}
 \smallskip

 
  
\begin{assumption}
\label{ass:3}
For all $i \in \N$, let $\mathcal{I}\QP(\dz\ii, z\ii)$ denote the index set of the active inequality constraints  in \eqref{eq:QP_z} parametric in $z\ii$, i.e.,
\begin{equation*}
\mathcal{I}\QP(\dz\ii, z\ii) := \Set{j\in \{ 1, \ldots, m \} | g_j(z\ii) + \nabla g_j (z\ii)^\top \dz\ii = 0}.
\end{equation*}
Then $\dz\ii$ is regular, i.e., the matrix made up of $\nabla h(z\ii)$ along with the columns $\nabla g_j (z\ii)$, $j \in \mathcal{I}\QP(\dz\ii, z\ii)$, has full column rank. Further, strict complementarity holds.

 \end{assumption}
  \smallskip
  
Note that this implies that the dual variables $\lambda\QP\ii, \nu\QP\ii$ in \eqref{eq:KKT_QP_z} are bounded and unique.
%
\smallskip
\begin{standing}
\label{ass:4}
For all $i \in \N$, $z\ii, \, z\ii + \dz\ii \in \Omega \subset \R^n$. The set $\Omega$ is compact.
\end{standing}

\smallskip
\begin{standing}
\label{ass:5}
The functions  $J$, $g$, $h$, and their first and second derivatives are uniformly bounded in norm in $\Omega$.
\end{standing}
\smallskip

  Several possible choices for the Hessian $ H\ii$ have been considered in the literature. 
 By setting $ H\ii$ as the Hessian of the Lagrangian of \eqref{eq:original_problem} and unit step size, local convergence to the desired $z^\star$ is achieved with a quadratic rate \cite{goodman1985newton, tapia1974stable, tapia1978quasi}. 
 
 Other choices make the computation of $H\ii$ less expensive, but deteriorate the convergence speed. See   \cite{coleman1990characterizations} for a general overview of superlinear convergence theorems for SQP methods.

 


To ensure global convergence to a critical point, step sizes $t\ii$ different from $1$ are employed in SQP, together with a merit function in the form of an augmented Lagrangian:
\begin{multline} \label{eq:augmented_lagrangian}
\al{z,\lambda,\nu, s} := \textstyle \ J(z) + (g(z) + s)^\top \lambda +  h(z)^\top  \nu \\
+ \frac{\rho}{2}  \norm{g(z)+s}^2+ \frac{\rho}{2}  \norm{h(z)}^2, 
\end{multline}
where $\rho\geq 0$ is a  penalty parameter to be determined and $s \in \Real_{\geq 0}^m$ is a vector of slack variables, defined at the beginning of each iteration $i$ such that its $j$th component satisfies the following equation \cite[Equation (2.8)]{gill_SQP_theory}:
\begin{equation}
\label{eq:slack_s}
s\ii_j := \begin{cases}
\max \left\{ 0,-g_j(z\ii) \right\} & \text{ if } \rho = 0 \smallskip \\
\textstyle \max \left\{ 0, -g_j(z\ii) - \frac{\lambda_j}{\rho} \right\} & \text{ otherwise}.
\end{cases}
\end{equation}
When the parameter $\rho$ is nonzero, the vector $s$ as defined in \eqref{eq:slack_s} yields the value of $\Lagr_\text{aug}$ minimized with respect to the slack variables alone, i.e., $\partial \Lagr_\text{aug} / \partial s = 0$, subject to the non-negativity constraint $s\geq 0$.

Several possibilities for the design of the dual variables $\lambda$ and $\nu$ have been considered in the literature. In \cite{Boggs_survey, powell1986recursive} a least square estimate, function of $z$ and dependent on the  Jacobian matrices of the objective and constraints, is proposed and global convergence properties are proved, although this approach is computationally expensive. This is partially alleviated in \cite{han1976superlinearly, powell1978fast}, where the multipliers $\lambda\QP$ and $\nu\QP$ are used as constant estimates of the dual variables, thereby reducing the computational burden. 
%
%
On the other hand, this has the effect of redefining the merit function and leads to theoretical difficulties when proving global convergence.

The approach we follow in this paper builds upon that in \cite{gill_SQP_theory}. Therein, $\lambda$ and $\nu$ are considered as additional variables, updated with step size $t\ii$ along with the primal sequence $\left( z\ii\right)_i$. Specifically, in view of \cite{gill_SQP_theory}, we consider the iterative update
\begin{equation}
\label{eq:update_dual}
\left[ \begin{array}{c}
z\ip \\ \lambda \ip \\ \nu\ip \\ s\ip  
\end{array} \right] := 
\left[ \begin{array}{c}
z\ii \\ \lambda \ii \\ \nu\ii \\ s\ii  
\end{array} \right]
+ t\ii \left[ \begin{array}{c}
\dz\ii \\ \dl \ii \\ \dn\ii \\ \ds\ii 
\end{array} \right],
\end{equation}
where $\dz\ip$ is from \eqref{eq:QP_z}, $\dl\ii := \lambda\QP\ii - \lambda\ii$,  $\dn\ii := \nu\QP\ii - \nu\ii$ 
and the slack variation $\ds\ii$ satisfies
\begin{equation}
\label{eq:ds_def}
g(z\ii) + \nabla g (z\ii)^\top \dz\ii + s\ii + \ds\ii = 0.
\end{equation}

Then we define the function $\phi: \R \rightarrow \R$ as
\begin{equation}
\label{eq:def_phi}
\phi (t) := \al{z + t \dz,\lambda + t\dl,\nu + t\dn,s+t\ds}
\end{equation}
to determine the  step size $t$, e.g. via a backtracking line search starting from $t=1$, that satisfies the 
 Wolfe conditions \cite{dennis1996numerical, nocedal1992theory}:
\begin{subequations}
\label{eq:Wolfe_condition}
\begin{align}
&\textstyle \phi(t) - \phi(0) \leq  \ \sigma_1 \, t \, \phi'(0)  
\label{eq:Wolfe_condition_1}
\\
&\textstyle |\phi'(t)| \leq  -\sigma_2 \, \phi'(0)\ \textup{ or } \ \left( t=1 \textup{ and } \phi'(1) \leq - \sigma_2 \phi'(0) \right)  
\label{eq:Wolfe_condition_2}
\end{align}
\end{subequations} 
for some  $0 < \sigma_1 \leq \sigma_2 < \frac{1}{2}$. For ease of notation, let us avoid making explicit the dependence of $\phi$ on the arguments $z, \lambda, \nu, s$ of the augmented Lagrangian function $\mathcal{L}_{\textup{aug}}$.

Note that if the derivative $\phi'(0)$ is negative, then there exists a step size $t\ii \in (0,1]$ such that the conditions in \eqref{eq:Wolfe_condition} hold. 
The condition on the derivative $\phi'(0)$ is checked numerically at every iteration $i \in \N$ via the inequality condition
\begin{equation}
\label{eq:condition_eta_SQP}
\phi'(0) \leq -\frac12 (\dz\ii )^\top  H\ii \dz\ii.
\end{equation}
If this latter inequality does not hold true, then the parameter $\rho$ is adjusted. In particular, there exists a lower bound $ \underline{\rho} \in \R_{\geq 0}$ such that the inequality in \eqref{eq:condition_eta_SQP} holds for all $\rho \geq \underline{\rho}$  \cite[Lemma 4.3]{gill_SQP_theory}.

For the practical implementation of the algorithm, the line search in \eqref{eq:Wolfe_condition} is typically simplified in order to check only the first condition in \eqref{eq:Wolfe_condition_1} \cite{dennis1977quasi, nocedal1999numerical}. This has the effect of reducing the computational burden required to compute the derivative $\phi'(t)$ and it does not impede convergence of the algorithm in practice. To derive the step-size $t$, a backtracking line-search is employed with safeguarded polynomial interpolation \cite{lemarechal1981view}.
   
The SQP steps for the NLP in \eqref{eq:original_problem} are summarized in \Cref{alg:SQP}.

\begin{algorithm}[t]
\caption{Sequential Quadratic Programming}   
\label{alg:SQP}
    \begin{algorithmic} 
            \State \textsc{Initialize} $i \leftarrow 0$ and $z^{(0)} \in \R^n$
            \Repeat
				\State \textsc{Compute}  $\dz\ii$ as in \eqref{eq:QP_z} and $\lambda\QP\ii$,  $\nu\QP\ii$ such that \eqref{eq:KKT_QP_z} holds
                \If{ $\dz\ii = 0$}  
                \State
                 \textsc{Set} $z^\star = z\ii$,  $\lambda^\star = \lambda\QP\ii$, $\nu^\star = \nu\QP\ii$ and \textsc{Stop}
                 \Else
                 \smallskip
                 \If{ $i = 0$}
                 \State{\textsc{Set}  $\lambda^{(0)} = \lambda\QP^{(0)}$, $\nu^{(0)} = \nu\QP^{(0)}$}
                 \EndIf 
                 \smallskip
                 \State{\textsc{Set}  $\dl\ii = \lambda\QP\ii - \lambda\ii$, $\dn\ii = \nu\QP\ii - \nu\ii$}
                 \smallskip
                \EndIf				
                \smallskip
				
				\State \textsc{Determine} $s\ii$ and $\ds\ii$ from \eqref{eq:slack_s} and \eqref{eq:ds_def} 

\smallskip 
				
				\State \textsc{Set} $\rho \geq 0$ such that \eqref{eq:condition_eta_SQP} holds
				
				\smallskip
				\State \textsc{Determine} the step size $t\ii$ that satisfies \eqref{eq:Wolfe_condition}, 
				e.g. via line search 
				\smallskip
\State \textsc{Update} $z\ip, \lambda\ip, \nu\ip, s\ip$ as in \eqref{eq:update}
\smallskip
 \State $i \leftarrow i + 1$
 \smallskip
 \Until{\textit{Convergence}}
      \State \textbf{return} $z^\star$, $\lambda^\star$ and $\nu^\star$ 
    \end{algorithmic}
\end{algorithm}

\section{Proposed variant to Sequential Quadratic Programming}

\label{sec:Prp_Method}

The SQP method presented in Section \ref{sec:Iterative_Methods} requires the computation of the primal and dual optimal solutions to each of the QPs in \eqref{eq:QP_z}. Therefore, for all $i \in \N$, each QP has to be exactly solved to determine the updates $\dz\ii$, $\dl\ii$ and $\dn\ii$. 


In this paper, we propose determining $\dz\ii$ through one projected gradient step  onto the linearization of the constraints around $z\ii$, i.e. the feasible set of the QP in \eqref{eq:QP_z}. This is formalized by
\begin{equation} 
\label{eq:algorithm_dz}
\begin{split}
\dz\ii &:=  \Proja{\C\ii}{- \alpha\ii  \nabla J ( z\ii ) }, \\
\end{split}
\end{equation}
with bounded gradient step size $\alpha\ii \in \R_{>0}$, and $\Proja{\C\ii}{\cdot} : \R^n \rightarrow \C\ii \subseteq \R^n$ being the Euclidean projection onto the set
\begin{equation}
\label{eq:def_Ci}
\textstyle \C\ii := \Set{ \dz \in \Real^n | g(z\ii) + \nabla g (z\ii)^\top \dz \leq 0 , \,  h(z\ii) + \nabla h (z\ii)^\top \dz = 0 }.
\end{equation}
Note that the algorithm step in \eqref{eq:algorithm_dz} is equivalent to computing only one gradient step of the QP in \eqref{eq:QP_z} with null initialization $\dz^\text{ini}$:
$$
\dz\ii = \Proja{\C\ii}{ \underbrace{\dz^\text{ini}}_{=0} - \alpha\ii\left( \underbrace{H\ii \dz^\text{ini}}_{=0} +  \nabla J ( z\ii ) \right) } = \Proja{\C\ii}{ - \alpha\ii  \nabla J ( z\ii )  }.
$$
In Section \ref{sec:Proof}, more detail is given about the choice of $\alpha$ for the convergence of the algorithm.
%
Note that the projection in \eqref{eq:algorithm_dz} is equivalent to solving the following optimization problem.
\begin{equation}
\label{eq:proj_opt}
\begin{split}
\dz\ii = \arg \min_{\dz \in \R^n} \ \  &\frac{1}{2 \alpha\ii} \norm{\dz + \alpha\ii \nabla J(z\ii)}^2 \\ \text{s.t.} \ \ & g(z\ii) + \nabla g (z\ii)^\top \dz \leq 0 \\ & h(z\ii) + \nabla h (z\ii)^\top \dz = 0.
\end{split}
\end{equation}
Therefore, there exist dual multipliers $\lambda\G\ii \in \R^m_{\geq 0}$, $\nu\G\ii \in \R^p$ such that
\begin{equation}
\label{eq:dz_KKT}
\begin{split}
& \frac{1}{\alpha\ii} \dz\ii + \nabla J (z\ii) + \nabla g(z\ii) \lambda\G\ii + \nabla h(z\ii) \nu\G\ii  = 0 \\
& \diag(\lambda\G\ii) \left( g(z\ii) + \nabla g (z\ii)^\top \dz\ii \right) = 0 \\
&h(z\ii) + \nabla h(z\ii)^\top \dz\ii = 0. 
\end{split}
\end{equation}
The dual variables increments then are
\begin{equation}
\label{eq:dual_alg}
\begin{split}
\dl\ii &:= \lambda\G\ii -  \lambda\ii \\
\dn\ii &:= \nu\G\ii - \nu\ii,
\end{split}
\end{equation}
while we define the slack variable $s\ii$ and variation $\ds\ii$ as in the SQP method, i.e., from \eqref{eq:slack_s} and \eqref{eq:ds_def}, respectively.

Analogously to the SQP method reviewed in \Cref{sec:SQP_Method}, the augmented Lagrangian function in the form \eqref{eq:augmented_lagrangian} is considered. The dual variables are updated along with the primal variables according to the update equation in \eqref{eq:update_dual}, with step size $t\ii \in (0,1]$ defined such that  the conditions in \eqref{eq:Wolfe_condition} hold. 

We choose the penalty parameter $\rho \in \R_{\geq 0}$ such that the condition
\begin{equation}
\label{eq:condition_phi_PM}
\phi'(0) \leq -\frac{1}{2 \alpha\ii} \norm{\dz\ii}^2
\end{equation}
is satisfied with $\phi(\cdot)$ defined as in \eqref{eq:def_phi}. Later, in \Cref{lemma:tune_eta}, we provide a lower bound  $\hat \rho$ such that \eqref{eq:condition_phi_PM} holds for all $\rho \geq \hat \rho$.

Our proposed approach is summarized in Algorithm \ref{alg:nonlinear_algorithm}.

\begin{algorithm}[t]
\caption{Variant to SQP}   
\label{alg:nonlinear_algorithm}
    \begin{algorithmic} 
            \State \textsc{Initialize} $i \leftarrow 0$ and $z^{(0)} \in \R^n$
            \Repeat
            
				\State \textsc{Compute} $\dz\ii$ with step size $\alpha\ii$ as in \eqref{eq:algorithm_dz} 
				\smallskip
				
				\State \textsc{Determine} $\lambda\G\ii$, $\nu\G\ii$ such that \eqref{eq:dz_KKT} holds
				\smallskip
	
				\smallskip 
				                \If{ $\dz\ii = 0$}  
                \State
                \textsc{Set} $z^\star = z\ii$,  $\lambda^\star = \lambda\G\ii$, $\nu^\star = \nu\G\ii$ and \textsc{Stop}
                 \Else
                                  \If{ $i = 0$}
                 \State{\textsc{Set}  $\lambda^{(0)} = \lambda\G^{(0)}$, $\nu^{(0)} = \nu\G^{(0)}$}
                 \EndIf 
                 \smallskip
                 \State{\textsc{Set}  $\dl\ii = \lambda\G\ii - \lambda\ii$, $\dn\ii = \nu\G\ii - \nu\ii$}
                 \smallskip
                \EndIf	
\smallskip						
				\State \textsc{Determine} $s\ii$ and $\ds\ii$ from \eqref{eq:slack_s} and \eqref{eq:ds_def} 

\smallskip 
				
				\State \textsc{Set} $\rho \geq 0$ such that \eqref{eq:condition_phi_PM} holds
				\smallskip
				\State \textsc{Determine} the step size $t\ii$ that satisfies \eqref{eq:Wolfe_condition}, e.g. via line search 
				\smallskip
\State \textsc{Update} $z\ip, \lambda\ip, \nu\ip, s\ip$ as in \eqref{eq:update}
\smallskip
 \State $i \leftarrow i + 1$
 \smallskip
 \Until{\textit{Convergence}}
      \State \textbf{return} $z^\star$, $\lambda^\star$ and $\nu^\star$ 
    \end{algorithmic}
\end{algorithm}


\section{Proof of convergence of the proposed algorithm}
\label{sec:Proof}
In this section, we show the convergence properties of the proposed approach in Algorithm \ref{alg:nonlinear_algorithm}, under the standing assumptions of the SQP method in \Cref{sec:SQP_Method} and the following assumption that replaces \Cref{ass:3}.
\smallskip
\begin{assumption}
\label{ass:6}
For all $i \in \N$, and $z\ii$, let $\mathcal{I}\G(\dz\ii, z\ii)$ denote the index set of the active constraints in \eqref{eq:proj_opt}, i.e.,
\begin{equation} \label{eq:our-active-set}
\mathcal{I}\G(\dz\ii, z\ii) = \Set{j \in \{ 1, \ldots, m \} | g_j(z\ii) + \nabla g_j (z\ii)^\top \dz\ii = 0}.
\end{equation}
Then $\dz\ii$ is regular, i.e., the matrix made up of $\nabla h(z\ii)$ along with the columns $\nabla g_j (z\ii)$, ${\forall} j \in \mathcal{I}\G(\dz\ii, z\ii)$, has full column rank. Furthermore, strict complementarity holds. {\hfill $\square$}

 \end{assumption}
\smallskip
{Note that Assumption \ref{ass:6}}
implies that the dual variables $(\lambda\G\ii,\nu\G\ii )$ in \eqref{eq:dz_KKT} are bounded and unique.
{Also note} that the problem in \eqref{eq:proj_opt} has the same feasible set as \eqref{eq:QP_z}, thus by \Cref{ass:2}, the projection in \eqref{eq:algorithm_dz} is always feasible.

This section is organized as follows: in \Cref{sec:Proof_local}, we derive conditions for local linear convergence (setting $t=1$ in \eqref{eq:Wolfe_condition}) under additional assumptions on the Hessian at the critical point and on the {step size} $\alpha$. 
These conditions 	are not required
to prove the global convergence of the algorithm in \Cref{sec:Proof_global}, albeit for $t \leq 1$ the convergence can be theoretically slower {than linear}. We finally show that the linear convergence rate is not precluded close to the solution, since $t=1$ is admissible by the line search in \Cref{sec:t_1}.

\subsection{Local convergence}
\label{sec:Proof_local}

According to \cite{day1963recursive, Robinson}, we define the general recursive algorithm as a method to determine {a critical point for} the NLP in \eqref{eq:original_problem} via intermediate iterates $w\ii := \left( z\ii, \lambda\ii, \nu\ii \right)$, whose update $w\ip$ is determined as the KKT triple of a specific optimization problem $\text{P}(w\ii)$. Given the following generic optimization problem:
\begin{equation}
\label{eq:general_recursive_algorithm}
\begin{split}
\text{P}(w\ii) : \quad z\ip = \arg \min_{z} \ &\J(z,w\ii) \\
 					\text{s.t.}\ &\g(z,w\ii)\leq 0 \\
								&\h(z,w\ii) = 0,
\end{split}								
\end{equation}
the updates $\lambda\ip$ and $\nu\ip$ are the dual variables associated {with} the KKT conditions for P$(w\ii)$. 
As in \cite{Robinson}, and in line with the original NLP in \eqref{eq:original_problem}, 
we assume that the functions $\J$, $\g$ and $\h$ are twice continuously differentiable in their first argument.
Let us define a KKT triple as $w:=\left(z,\lambda,\nu\right)$ and the function
\begin{multline}
\label{eq:define_U}
\mathcal{U} \left( w, \, w\ii  \right) := \left[ \nabla_{\!w} \, \J(z,w\ii) + \lambda^\top \nabla_{\!w} \, \g (z,w\ii) + \nu^\top \nabla_{\!w} \, \h(z,w\ii); \right. \\ \left. \lambda_1 \g_1 (z,w\ii) ; \ldots ;  \lambda_m \g_m (z,w\ii); \h_1 (z,w\ii); \ldots ; \h_p (z,w\ii)  \right].
\end{multline}

The Sequential Quadratic Programming methods and the proposed algorithm can be recast as general recursive algorithms of the form in \eqref{eq:general_recursive_algorithm}. Let us first consider the SQP in \Cref{alg:SQP}. A local version of this algorithm, which is not guaranteed to converge for any initialization $z^{(0)}$, takes step $t\ii = 1$ for all $i \in \N$.


It follows from \eqref{eq:QP_z} that the update $w\ip = (z\ip, \lambda\ip, \nu\ip)$ associated with \eqref{eq:general_recursive_algorithm} is the KKT triple of the problem $\text{P}\QP(w\ii)$:
\begin{equation} \label{eq:P-QP}
\begin{split}
\text{P}\QP(w\ii) : \quad z\ip = \arg \min_{z} \ & \frac12  (z - z\ii)^\top H\ii (z - z\ii) + \nabla J(z\ii)^\top (z - z\ii) \\
 					\text{s.t. }\ &g(z\ii) + \nabla g(z\ii)^\top (z - z\ii) \leq 0 \\
								&h(z\ii) + \nabla h(z\ii)^\top (z - z\ii) = 0.
\end{split}	
\end{equation}
In fact, since $t\ii = 1$, by \eqref{eq:update_dual} and \eqref{eq:dual_alg}, the update of the dual variables is given by  
$\left( \lambda\ip, \, \nu\ip \right)=(\lambda\QP\ii, \, \nu\QP\ii )$.

Analogously, if we fix $t\ii = 1$ for all $i \in \N$, it follows from \eqref{eq:proj_opt} that the proposed algorithm determines the update $w\ip$ as the KKT triple of {the problem} $\text{P}\G(w\ii)$:
\begin{equation} \label{eq:P-G}
\begin{split}
\text{P}\G(w\ii) : \quad z\ip = \arg \min_{z} \ & \frac{1}{2 \alpha\ii}  (z - z\ii)^\top (z - z\ii) + \nabla J(z\ii)^\top (z - z\ii) \\
 					\text{s.t. }\ &g(z\ii) + \nabla g(z\ii)^\top (z - z\ii) \leq 0 \\
								&h(z\ii) + \nabla h(z\ii)^\top (z - z\ii) = 0.
\end{split}	
\end{equation}
Again, by \eqref{eq:update_dual} and \eqref{eq:dual_alg}, the update of the dual variables is 
$\left( \lambda\ip, \, \nu\ip \right) = ( \lambda\G\ii, \, \nu\G\ii)$.

The following two results establish some basic properties of the general recursive algorithm in \eqref{eq:general_recursive_algorithm} that will be necessary to establish local and global convergence of the proposed algorithm. 



\begin{lem}[{\cite[Theorem 2.1]{Robinson}}]
\label{th:gen_rec_alg}
Let $\bar w~\in~\R^{n + m + p}$, and suppose that $\left( \bar{z}, \bar \lambda, \bar \nu \right) \in \R^n \times \R^m \times \R^p$ is a KKT triple of P$(\bar w)$ from \eqref{eq:general_recursive_algorithm}, at which the first order conditions with strict complementarity slackness hold.

Then, there exist open neighborhoods $W = W(\bar{w})$ and $V = V(\bar{z}, \bar \lambda, \bar \nu)$, and a continuous function $Z : W \rightarrow V$, such that: $Z(\bar w) = (\bar z, \bar \lambda, \bar \nu)$, for all $w \in W$, $Z(w)$ is the unique KKT triple in $V$ of P$(w)$ and the unique zero in $V$ of the function $\mathcal{U}\left((\cdot), \, w \right)$ in \eqref{eq:define_U}. 
Furthermore, if $Z(w) =: (z(w),\lambda(w),\nu(w))$, then for each $w \in W$, $z(w)$ is a critical point of P$(w)$ at which the first order KKT conditions are satisfied with strict complementarity slackness and linear independence of the gradients to the active constraints.  
{\hfill $\square$}
\end{lem}

Specifically, the same inequality constraints active at $z(\bar w)$ will be active at $z(w)$, which is in accordance with \cite[Proof of Theorem 2.1]{Robinson}. However, since we assume \textit{first order} conditions with slack complementarity at $z(\bar w)$, then at $z(w)$ the \textit{first order} conditions will hold with slack complementarity.

As in \cite{Boggs_survey}, for the analysis of local convergence we assume that the correct active set at $z^\star$ is known. This is justified by  \Cref{th:gen_rec_alg}, since the proposed algorithm will eventually identify the active inequality constraint for \eqref{eq:original_problem}. Therefore, we define as $h_\txa(z) = [g_\txa (z) \, ; \, h (z) ]$ the set of the active inequality and equality constraints at $z^\star$, and indicate with $\xi :=  [\lambda_\txa \, ; \, \nu ]$ the corresponding dual variables. 

The following result establishes the local convergence properties of the algorithm and apply SQP arguments to establish linear convergence to a critical point. In fact, the problem in \eqref{eq:P-G} can be seen as an SQP program with Hessian $\frac{1}{\alpha} I \succ 0$.

\begin{theorem}
\label{th:local_conv}
Assume that $\left(z^\star,\xi^\star\right)$ is a critical point such that $\nabla^2 \Lagr(z^\star,\xi^\star)$ is positive definite, and let the initialization $z^{(0)}$ be close enough to $z^\star$. Then, there exist positive step sizes 
$(\alpha\ii)_i$ such that the sequence $\left(z\ii\right)
_i$ {defined as in \eqref{eq:P-G} converges to $z^\star$ with linear rate.}
{\hfill $\square$}
\end{theorem}

\smallskip 

\begin{proof}
The proof is similar to \cite[Proof of Theorem 3.3]{boggs1982linear}, albeit the result obtained is different. 
For ease of notation, we use no superscript for the iteration $(i)$ and the superscript ``$+$'' for the iteration $(i+1)$.

Then the proposed algorithm has the following update when $t=1$:
\begin{equation*}
\begin{split}
\xi^{+} &= \xi + \dxi = \xi + \left( \xi\G (z) - \xi \right) = \xi\G (z) =  \\
&= \left( \alpha \nabla h_\txa (z)^\top \nabla h_\txa (z) \right)^{-1} \left( h_\txa(z) - \alpha \nabla h_\txa (z) \nabla J(z)  \right). 
\end{split}
\end{equation*}
From \eqref{eq:FOC}, {we have the} the optimal dual variable 
\begin{equation*}
\begin{split}
\xi^{\star} &= - \left( \nabla h_\txa (z^{\star})^\top \, A \, \nabla h_\txa (z^{\star}) \right)^{-1} \left( \nabla h_\txa (z^{\star}) \, A \, \nabla J(z^{\star})  \right),  
\end{split}
\end{equation*}
where $A$ is any nonsingular matrix that is positive definite on the null space of $\nabla h_\txa (z^{\star})^\top$. 
{In particular, with} $A = \alpha I$ {we obtain}
\begin{equation}
\label{eq:xi_minus_xistar}
\begin{split}
&\xi^{+} - \xi^\star = \xi\G (z) - \xi\G (z^\star) = \nabla \xi\G (z^\star) (z - z^\star) + \mc{O}\left(\norm{z - z^\star}^2 \right) \\
&= \! \left( \alpha \nabla h_\txa (z^{\star})^\top \nabla h_\txa (z^{\star}) \right)^{-1}\! \alpha \nabla h_\txa \! (z^{\star})^\top  \left(\frac{1}{\alpha} I - \nabla^2 \Lagr (z^\star, \xi^\star ) \right)  (z - z^\star) \! + \! \mc{O}\left(\norm{z - z^\star}^2 \right)\!. 
\end{split} 
\end{equation}
From \eqref{eq:dz_KKT}, {we have} the update $
\dz = - \alpha \nabla \Lagr (z, \xi\G (z) )
$, therefore
\begin{equation}
\label{eq:z_minus_zstar}
\begin{split}
z^{+} - z^\star &= z + \dz - z^\star = z - z^\star - \alpha (\nabla \Lagr (z, \xi\G (z)) - \underbrace{\nabla \Lagr (z^\star, \xi^\star)}_{=0} ) \\
&= z - z^\star  - \alpha \left( \nabla^2 \Lagr (z^\star,\xi^\star) (z - z^\star) + \nabla h_\txa(z^\star) ( \xi\G(z) - \xi^{\star} )  \right) + \mc{O}\left(\norm{z - z^\star}^2\right) \\
&=\! \alpha \left( \! \left(\frac{1}{\alpha} I  -  \nabla^2 \Lagr (z^\star,\xi^\star) \right) \! (z - z^\star) - \! \nabla h_\txa(z^\star) ( \xi\G(z) - \xi^{\star} )  \right) \! + \mc{O}\left(\norm{z - z^\star}^2\right) \! . \\
\end{split}
\end{equation}

{Now,} by substituting \eqref{eq:xi_minus_xistar} into \eqref{eq:z_minus_zstar}
\begin{equation*}
\begin{split}
z^{+} - z^\star
&= \alpha \left(  \left(\frac{1}{\alpha} I  -  \nabla^2 \Lagr (z^\star,\xi^\star) \right) (z - z^\star) - \nabla h_\txa(z^\star)  \left( \alpha \nabla h_\txa (z^{\star})^\top \nabla h_\txa (z^{\star}) \right)^{-1} \cdot \right. \\ & \quad \left. \cdot \, \alpha \nabla h_\txa (z^{\star})^\top  \left(\frac{1}{\alpha} I - \nabla^2 \Lagr (z^\star, \xi^\star ) \right)  (z - z^\star)  \right) + \mc{O}\left(\norm{z - z^\star}^2\right)  \\
&= \alpha T(z^\star) \left(\frac{1}{\alpha} I  -  \nabla^2 \Lagr (z^\star,\xi^\star) \right) (z - z^\star) + \mc{O}\left(\norm{z - z^\star}^2\right) \\
&=  T(z^\star) \left(I  - \alpha \nabla^2 \Lagr (z^\star,\xi^\star) \right) (z - z^\star) + \mc{O}\left(\norm{z - z^\star}^2\right),
\end{split}
\end{equation*}
with $T(z^\star) := I -  \nabla h_\txa (z^\star) \left( \nabla h_\txa (z^\star)^\top \nabla h_\txa (z^\star) \right)^{-1} \nabla h_\txa (z^\star)^\top$ being the orthogonal projector onto the tangent space to the constraints $h_a(z^\star)$ at $z^\star$.

{Then by considering the norms of the above quantities we conclude that}
\begin{equation*}
\begin{split}
\norm{z^{+} - z^\star} & \leq \norm{T(z^\star) \left(I  - \alpha \nabla^2 \Lagr (z^\star,\xi^\star) \right) (z - z^\star)} + \gamma \norm{z - z^\star}^2 \\
& \leq \big( \underbrace{\norm{I  - \alpha \nabla^2 \Lagr (z^\star,\xi^\star)}  + \gamma \norm{z - z^\star}}_{=:\eta} \big) \norm{z - z^\star},
\end{split}
\end{equation*}
for all sufficiently large iterations and some $\gamma>0$, independent of the iteration. {Here we have used} the property of the orthogonal projector $\norm{T(z^\star) \, v} \leq \norm{v}$ for any vector $v$. 
Since $\nabla^2 \Lagr (z^\star,\xi^\star) \succ 0$, 
by choosing $\alpha \leq \frac{1}{\max \text{eig} \, \nabla^2 \Lagr (z^\star,\xi^\star)}$, the term $\norm{I  - \alpha \nabla^2 \Lagr (z^\star,\xi^\star)}$ can be made strictly smaller than $1$. Thus, for a {sufficiently} small initialization distance $\norm{z^{(0)} - z^\star}$,  {we have} $\eta < 1$, and the sequence $\left(z\ii\right)_i$ converges {at a linear rate due to the contraction mapping theorem}.  
\end{proof}

\smallskip

\subsection{Global convergence}
\label{sec:Proof_global}
Before showing the convergence result of the paper, some technical lemmas are presented that show the properties of the proposed algorithm. In particular, the following Lemma ensures that the desired algorithm determines the correct active set at a critical point of \eqref{eq:original_problem}.
\begin{lem}
\label{lemma:active_set}
The following properties hold for \Cref{alg:nonlinear_algorithm}:
\begin{enumerate}
\item[\textup{(i)}]  $\textstyle\norm{\dz\ii} = 0$ if and only if $z\ii$ is a critical point for \eqref{eq:original_problem};
\item[\textup{(ii)}] there exists $\bar \epsilon \in \R_{>0}$ such that: if $\norm{\dz}\leq \bar{\epsilon}$, then the active set $\mathcal{I}\G$ in \eqref{eq:our-active-set} of \eqref{eq:proj_opt} coincides with the set of constraints that are active at a critical point $z^\star$ for \eqref{eq:original_problem}.
\end{enumerate}
{\hfill $\square$}
\end{lem}

\begin{proof}
\textup{(i)} We first prove that, if $\norm{\dz\ii} = 0$, then $z\ii$ is a KKT point for \eqref{eq:original_problem}. Note that $\norm{\dz\ii} = 0$ implies $\dz\ii = 0$, i.e., that there exist $\lambda\G\ii$ and $\nu\G\ii$ such that \eqref{eq:dz_KKT} holds with $\dz\ii = 0$. Therefore, by setting $z^\star = z\ii$,  $\lambda^\star = \lambda\G\ii$ and $\nu^\star = \nu\G\ii$, the KKT conditions in \eqref{eq:FOC} are satisfied. Strict complementarity follows from \Cref{ass:6}. 

Conversely, if $z\ii$ is a critical point for \eqref{eq:original_problem}, then \eqref{eq:FOC} holds for  $z^\star = z\ii$,  $\lambda^\star = \lambda\G\ii$ and $\nu^\star = \nu\G\ii$. Now, suppose that the vector $\dz\ii$, resulting from the projection in \eqref{eq:algorithm_dz} is nonzero, i.e., there exist $\lambda\G\ii \in \R^m_{\geq 0}$ and $\nu\G\ii \in \R^p$ such that \eqref{eq:dz_KKT} hold for some $\dz\ii \neq 0$, with strict complementarity by \Cref{ass:6}. On the other hand, because of \eqref{eq:FOC}, the KKT conditions in \eqref{eq:dz_KKT} hold also for $\dz = 0$, with dual variables $\bar \lambda\G\ii \in \R^m_{\geq 0}$ and $\bar \nu\G\ii \in \R^p$. 
{This generates a contradiction,} because the projection  \eqref{eq:algorithm_dz} onto the convex set $\C\ii$, which is nonempty by \Cref{ass:2}, is unique. 

\textup{(ii)} This result follows from \Cref{th:gen_rec_alg} since the proposed method can be rewritten as a general recursive algorithm in the form of \eqref{eq:general_recursive_algorithm}.
\end{proof}

\begin{lem}
\label{lemma:dual_bounded}
For all $i\in \N$, it holds that
\begin{equation*}
\norm{\lambda\ip} \leq \max_{ k \in [0,i]} \norm{\lambda\G^{(k)}}, \
\norm{\nu\ip} \leq \max_{ k \in [0,i] } \norm{\nu\G^{(k)}}.
\end{equation*}
In addition, $\norm{\dl\ii}$ and $\norm{\dn\ii}$ are uniformly bounded for all $i \in \N$.
{\hfill $\square$}
\end{lem}

\begin{proof}
First note that the dual variables $\lambda\G\ii$ and $\nu\G\ii$ defined in \eqref{eq:dz_KKT} are bounded in norm, since by \Cref{ass:6} the active set is linearly independent and 
strong duality holds. 

Then the proof follows the same argument of \cite[Lemma 4.2]{gill_SQP_theory}, where the structure of \eqref{eq:dual_alg} is exploited. We equivalently define the dual variable for the inequality constraints as 
\begin{equation}
\label{eq:dual_def_2}
\begin{split}
\lambda^{(0)} &:= \lambda\G^{(0)} \\
\lambda\ip    &:= \lambda\ii + t\ii ( \lambda\G\ii - \lambda\ii) \quad \forall i \in \N.
\end{split}
\end{equation}
We proceed by induction. The result holds for $\lambda^{(0)}$. Now we assume that the {result} holds for $\lambda\ii$. Then, since $t\ii \in (0,1]$ we have that
\begin{equation*}
\begin{split}
 \norm{\lambda\ip} &=  t\ii  \norm{\lambda\G\ii}  +  (1 - t\ii) \norm{ \lambda\ii} \\ 
  & \leq 
  t\ii \norm{\lambda\G\ii}  +  (1 - t\ii) \max_{ k \in [0,i-1]} \norm{\lambda\G^{(k)}} \\
  & \leq t\ii  \max_{ k \in [0,i]} \norm{\lambda\G^{(k)}} +  (1 - t\ii) \max_{ k \in [0,i]} \norm{\lambda\G^{(k)}} \\
 & = \max_{ k \in [0,i]} \norm{\lambda\G^{(k)}}.
 \end{split}
\end{equation*}
This proves the boundedness of $\norm{\lambda\ii}$, since both $\lambda_{\textup{G}}\ii$ and $\dl\ii$ are bounded 
{by \eqref{eq:dual_alg}}.


The proof for the boundedness of the dual variables associated to the equality constraints is analogous.
\end{proof}


%

The following Lemma serves as a tuning rule for the parameter $\rho$.
\smallskip
\begin{lem}
\label{lemma:tune_eta}
There exists $\hat \rho \in \R_{\geq 0}$ such that
\begin{equation*}
\begin{split}
& \sup_{\rho \geq \hat\rho}  \phi' (0, \rho) \leq -\frac{1}{2\alpha\ii} \norm{\dz\ii}^2,
\end{split}
\end{equation*}
where $\phi$ is as in \eqref{eq:def_phi} and $\alpha\ii$ and $\dz\ii$ are from \eqref{eq:algorithm_dz}.
{\hfill $\square$}
\end{lem} 
\smallskip

\begin{proof}
For a given $\rho \in \mathbb{R}_{\geq 0}$, the gradient of $\mathcal{L}_{\textup{aug}}$ in \eqref{eq:augmented_lagrangian} is	
\begin{equation*}
\left[ \begin{matrix}
\nabla J(z) + \nabla g(z)  \lambda + \nabla h(z) \nu + \rho \nabla g(z) (g(z) + s) +  \rho \nabla h(z) \, h(z)\\  g(z) + s \\ h(z) \\ \lambda + \rho \, ( g(z) + s)
\end{matrix} \right].
\end{equation*}
Therefore, using a simplified notation, we have
\begin{equation*}
\begin{split}
\phi'(0) = \ &\dz^\top \left( \nabla J + \nabla g \,  \lambda + \nabla h \, \nu + \rho \nabla g \, \left(g + s\right) + \rho \nabla h \, h \right)  \\ 
& + \dl^\top (g + s)+ \dn^\top h + \ds^\top \left( \lambda + \rho ( g + s) \right) \\
= \ & \dz^\top \nabla J 
+ ( \nabla g^\top \dz + \ds)^\top \lambda + \rho ( \nabla g^\top \dz + \ds)^\top (g + s)\\ 
& + \dz^\top \nabla h \nu + \dl^\top (g+s) + \dn^\top h + \rho (\nabla h^\top \dz)^\top h \\
= \ & \dz^\top \nabla J 
- ( g +s )^\top (\lambda - \dl) - \rho ( g + s)^\top (g+s)  \\ 
& - h^\top (\nu - \dn) - \rho h^\top h, 
\end{split}
\end{equation*}
where the last step follows from \eqref{eq:ds_def} and the last equation in \eqref{eq:dz_KKT}, as these imply that $\dz$ satisfies
\begin{subequations}
\label{eq:auxiliary_terms}
\begin{align}
\label{eq:auxiliary_terms_1} & \nabla g ^\top \dz + \ds = - (g + s) \\
\label{eq:auxiliary_terms_2} & \nabla h ^\top \dz = -h.
\end{align}
\end{subequations}
{By} rearranging the first equation in \eqref{eq:dz_KKT} and {using} the definition of $\lambda\G$ and $\nu\G$ {we have}
\begin{equation*}
\nabla J = -\frac{1}{\alpha} \dz - \nabla g \lambda\G - \nabla h \nu\G,
\end{equation*}
that, combined with  \eqref{eq:auxiliary_terms_1}, yields:
\begin{equation*}
\begin{split}
\phi'(0) = - &\frac{1}{\alpha} \dz^\top \dz - \dz^\top \nabla g \, \lambda\G - \dz^\top \nabla h \, \nu\G -  (g + s)^\top ( \lambda - \dl) \\ 
& -\rho ( g + s)^\top (g+s) - h^\top ( \nu - \dn) - \rho \, h^\top h \\
= \  - &\frac{1}{\alpha} \dz^\top \dz+ \ds^\top \lambda\G + (g + s)^\top \lambda\G  
- ( g +s )^\top (\lambda - \dl)  \\ 
&- \rho ( g + s)^\top (g+s) + h^\top \nu_G - h^\top (\nu- \dn) - \rho \, h^\top h.
\end{split}
\end{equation*}
By \eqref{eq:dual_alg} and the right-hand side of \eqref{eq:condition_phi_PM}, we then obtain   
\begin{equation*}
\begin{split}
 & \ds^\top \lambda\G + 2 (g + s)^\top \dl  - \rho \, ( g + s)^\top (g+s) + 2 h^\top  \dn - \rho \, h^\top h \\
 = \ &\ds^\top \lambda\G  + 2 \left[ \begin{matrix} (g+s)^\top &  h^\top \end{matrix} \right] \left[ \begin{matrix} \dl \\ \dn \end{matrix} \right] - \rho ( g + s)^\top (g+s) - \rho \, h^\top h \\
\leq \ &\frac{1}{2 \alpha} \norm{\dz}^2.
\end{split}
\end{equation*}
Note that $\ds^\top \lambda\G \leq 0$ because of \eqref{eq:ds_def} and the complementarity conditions in \eqref{eq:dz_KKT}. 
The determination of $\hat \rho$ is non-trivial only if
\begin{equation}
\label{eq:must_increase_eta_rho}
2 \left[ \begin{matrix} (g+s)^\top &  h^\top \end{matrix} \right] \left[ \begin{matrix} \dl \\ \dn \end{matrix} \right] \geq \frac{1}{2 \alpha}  \norm{\dz}^2, 
\end{equation} 
otherwise we can take $\hat \rho = 0$. Hence, if \eqref{eq:must_increase_eta_rho} holds, then we take $\hat \rho$ such that 
\begin{equation*}
\begin{split}
2 \left[ \begin{matrix} (g+s)^\top &  h^\top \end{matrix} \right] \left[ \begin{matrix} \dl \\ \dn \end{matrix} \right] \leq 2 \norm{ \left[ \begin{array}{c} g+s \\  h \end{array} \right] } \norm{ \left[ \begin{array}{c} \dl \\ \dn \end{array} \right]} \leq \hat \rho \norm{ \left[ \begin{array}{c} g+s \\  h \end{array} \right] }^2.
\end{split}
\end{equation*}
This is equivalent to
\begin{equation}
\label{eq:rho_hat}
\hat \rho = \textstyle 2 \, \left( {\textstyle\norm{ \left[
\begin{smallmatrix} \dl \\ \dn \end{smallmatrix} \right]
}}\right)/{\textstyle\norm{ 
\left[ \begin{smallmatrix} g + s \\ h \end{smallmatrix} \right]
}}.
\end{equation}
\end{proof}
\smallskip

Therefore, we can define the penalty parameter at the beginning of each iteration $i$ as follows.
\begin{equation}
\rho\ii := 
\begin{cases}
\rho\im & \text{if } \phi'(0,\rho\im) \leq - \frac{1}{2\alpha\ii} \norm{\dz\ii}^2 \\
\max \{\hat \rho\ii, \, 2 \rho\im\} & \text{otherwise},
\end{cases}
\end{equation}
where {$\hat \rho\ii = \hat{\rho}$ as in} \eqref{eq:rho_hat} with $d_{\lambda}$, $d_{\nu}$, $z$ and $s$ evaluated at the current iteration $i$.


\begin{rem}
Note that the parameter $\rho\ii$ can possibly diverge for $i\rightarrow \infty$, if there exists an infinite set of iterations $\{i_l\}_l$ where the parameter {strictly increases}. The statements given next consider this possibility and prove convergence in a general case.
\hfill $\square$
\end{rem}

%

\begin{lem}
\label{lemma:unbounded_rho}
Suppose $\left\{i_l \right\}_l$ is the set of iterations in which the penalty parameter $\rho$ increases. Then,
\begin{subequations}
\begin{align}
&\rho\ir \norm{\dz\ir}^2 \leq N_\rho
\label{eq:N_rho_1} \\
&\rho\ir \norm{ \left[ \begin{matrix} g(z\ir) + s\ir \\
h(z\ir)  \end{matrix} \right] } \leq N_\rho \label{eq:N_rho_2} 
\end{align}
\end{subequations}
for some {$N_\rho \in \R_{> 0}$.}
\end{lem}

\smallskip 
\begin{proof}
The argument of the functions and the index $i_\rho$ are dropped for ease of notation. In order for the penalty parameter to {increase}, {the} conditions in \eqref{eq:must_increase_eta_rho} must hold, {that is},
 \begin{equation*}
 \frac{1}{2 \alpha}  \norm{\dz}^2 \leq 2 \left[ \begin{matrix} (g+s)^\top &  h^\top \end{matrix} \right] \left[ \begin{matrix} \dl \\ \dn \end{matrix} \right] \leq 2 \norm{ \left[ \begin{array}{c} g+s \\  h \end{array} \right] } \norm{ \left[ \begin{array}{c} \dl \\ \dn \end{array} \right]},
 \end{equation*}
{hence}
  \begin{equation*}
 \norm{ \left[ \begin{array}{c} g+s \\  h \end{array} \right] } \geq \frac{1}{4 \alpha} \frac{\norm{\dz}^2}{\norm{ \left[ \begin{array}{c} \dl \\ \dn \end{array} \right]}}.
 \end{equation*}
By substituting this last inequality into the definition of $\hat \rho$, {we have that}
\begin{equation*}
 \hat \rho = \frac{2 \norm{ \left[
\begin{array}{c} \dl \\ \dn \end{array} \right]
}}{\norm{ 
\left[ \begin{array}{c} g+s \\ h \end{array} \right]
}} \leq 8 \alpha \frac{  \norm{ \left[
\begin{array}{c} \dl \\ \dn \end{array} \right] }^2}{\norm{\dz}^2},
\end{equation*}
and the desired result in \eqref{eq:N_rho_1} hold {due to} \Cref{lemma:dual_bounded}. The following relation proves  \eqref{eq:N_rho_2}:
\begin{multline}
\label{eq:Lemma9_intermediate}
\rho \norm{ 
\left[ \begin{array}{c} g+s \\ h \end{array} \right]
} \leq 2 \hat{\rho} \norm{ 
\left[ \begin{array}{c} g+s \\ h \end{array} \right]
} \\ = 2 \frac{2 \norm{ \left[
\begin{array}{c} \dl \\ \dn \end{array} \right]
}}{\norm{ 
\left[ \begin{array}{c} g+s \\ h \end{array} \right]
}} \norm{ 
\left[ \begin{array}{c} g+s \\ h \end{array} \right]
} = 4 \norm{ \left[
\begin{array}{c} \dl \\ \dn \end{array} \right]
}.
\end{multline}
\end{proof}

The following two Lemmas, provided without proofs, give intermediate technical results that are required for the main results. Their proofs follow the same arguments in \cite{gill_SQP_theory} with minor adjustments to reflect the update rule in \Cref{alg:nonlinear_algorithm}. 

\begin{lem}[{\cite[Lemma 4.6]{gill_SQP_theory}}]
\label{lemma:technical_unbounded}
Let $\{ i_l \}_l$ denote the set of iterations for which the parameter $\rho\ir$ increases.
Then, there exists $M \in \R_{>0}$ such that, for all $l\in \N$,
\begin{equation}
\label{eq:technical_bounded_2}
\rho\ir \sum_{i = i_l}^{i_{l+1} - 1} \norm{ t\ii \dz\ii}^2 < M.
\end{equation}
{$\hfill \square$}
\end{lem}

\begin{lem}[{\cite[Lemma 4.9]{gill_SQP_theory}}]
\label{lemma:phi_property}
The step size $t\ii$ defined according to \eqref{eq:Wolfe_condition} satisfies
$\phi(t\ii) - \phi(0) \leq \sigma_1 t\ii \phi'(0)$,
where $\sigma_1 < \frac12$ {and} $t\ii > \bar{t}$, for some $\bar{t}>0$ independent of $i$.
{$\hfill \square$}
\end{lem}
 
 
We are now ready to state the main result of the paper, that is, the global convergence of our proposed algorithm.

\begin{theorem}
\label{th:convergence}
\Cref{alg:nonlinear_algorithm} is such that
$ \displaystyle \lim_{i\rightarrow \infty} \norm{\dz\ii} = 0. $
{\hfill $\square$}
\end{theorem}

\begin{proof}
The proof is similar to \cite[Proof of Theorem 4.1]{gill_SQP_theory}. 
If $\norm{\dz\ii} = 0$ for a finite $i$, then the algorithm terminates and the statement is true. We assume in the following that $\norm{\dz\ii} \neq 0$ for all $i \in \N$. 

If there is no upper bound on $\rho$, then the uniform lower bound $\bar t>0$ {from} \Cref{lemma:phi_property} and \eqref{eq:technical_bounded_2} implies that, for all $\delta > 0$, there exists $\tilde i \in \N$ such that
$
\norm{\dz\ii} \leq \delta \text{ for all } i \geq \tilde{i},
$
{which} proves the statement. 

In the bounded case, there exists a value $\tilde \rho$ and an index $\tilde i$ such that $\rho = \tilde \rho$ for all $i \geq \tilde i$.
The proof is then by contradiction. We assume that there exist $\epsilon >0$ and $\tilde{i} \in \N$ such that $\norm{\dz\ii} > \epsilon$ for all $i \geq \tilde{i}$.
Now, every subsequent iteration must yield a decrease in the merit function in \eqref{eq:augmented_lagrangian} with $\rho = \tilde \rho$, since because of \eqref{eq:Wolfe_condition}, 
{\Cref{lemma:tune_eta} and} \Cref{lemma:phi_property}, we have
$$
\phi (t) - \phi(0) \leq \sigma_1 t \phi'(0) \leq - \frac{1 }{2 \alpha} \sigma_1 \bar{t} \varepsilon ^2 < 0.
$$
%
The addition of the slack variable $s$ in \eqref{eq:slack_s} can only lead to a further reduction in the merit function. Therefore, since the merit function with $\rho = \tilde \rho$ decreases by at least a fixed quantity at every iteration, it must be unbounded from below.
Since by \Cref{lemma:dual_bounded} the dual variables $\lambda$ and $\nu$ are bounded, the merit function in \eqref{eq:augmented_lagrangian} can be unbounded from below only if the objective, or the constraints functions, are unbounded from below. {This} leads to a contradiction, since due to Standing Assumptions \ref{ass:4} and \ref{ass:5} all the iterates lie in a region $\Omega$, where the objective and constraints functions are bounded in norm. 
Therefore, the result follows.
\end{proof}

\smallskip

\begin{theorem}
\label{th:primal_dual}
It holds that the primal and dual iterates in \Cref{alg:nonlinear_algorithm} converge to the KKT triple associated to a critical point $z^\star$ of \eqref{eq:original_problem}.  That is:
\begin{equation*}
\lim_{i\rightarrow \infty} \norm{z\ii - z^\star} =  \
\lim_{i\rightarrow \infty} \norm{\lambda\ii - \lambda^\star} =  \ 
\lim_{i\rightarrow \infty} \norm{\nu\ii - \nu^\star} = 0.
\end{equation*}
{\hfill $\square$}
\end{theorem}


\begin{proof}
The proof follows the line of \cite[Proofs of Corollary 4.1 and Theorem 4.2]{gill_SQP_theory}. See \Cref{app:proofs} for the proof details.
\end{proof}

\subsection{Asymptotic linear convergence}
\label{sec:t_1}
In this section we show that step sizes $t=1$ are not precluded by the Wolfe conditions in \eqref{eq:Wolfe_condition} when the iterates are sufficiently close to the solution. Therefore local convergence at a linear rate (\Cref{th:local_conv}) can be recovered.
The following standard SQP assumption is considered in the analysis \cite{gill_SQP_theory, powell1986recursive}.
\begin{assumption}
\label{ass:7}
For all sufficiently large $i$, the following holds:
\begin{equation*}
\begin{split}
z\ii + \dz\ii - z^\star &= o 	\left(\norm{z\ii - z^\star} \right) \\
\lambda\ii + \dl\ii - \lambda^\star &= o \left(\norm{\lambda\ii - \lambda^\star} \right) \\
\nu\ii + \dn\ii - \nu^\star &= o \left(\norm{\nu\ii - \nu^\star} \right) \\
  \left[ \dl\ii; \dn\ii\right]  &= \mc{O} \left( \norm{\dz\ii} \right).  \\
\end{split}
\end{equation*}
{\hfill $\square$}
\end{assumption}

\smallskip

This assumption implies that 
$$
\norm{ \dz\ii } \sim  \norm{z\ii - z^\star}, \, 
\norm{ \dl\ii } \sim  \norm{\lambda\ii - \lambda^\star}, \, 
\norm{ \dn\ii } \sim  \norm{\nu\ii - \nu^\star}
$$
{
where the notation ``$\sim$'' indicates that the quantities are of similar order as $i$ approaches infinity.
}
{Next we show} that the penalty parameter $\rho$ is bounded.
\begin{lem}
\label{lemma:rho_bounded}
{If} \Cref{ass:7} {holds}, {then} there exists a finite $\rho$ such that $\rho\ii \leq \bar \rho$ for all $i$.
{\hfill $\square$}
\end{lem}

\smallskip


\begin{proof}
The proof follows the same argument as \cite[Lemma 5.1]{gill_SQP_theory}.
Assume that the parameter $\rho$ is unbounded. Then, by \Cref{lemma:tune_eta}, the condition in \eqref{eq:must_increase_eta_rho} must hold over an infinite subsequence of iterations. Thus, using simplified notation,
\begin{equation*}
 \norm{\left[ \begin{matrix} g+s \\ h \end{matrix} \right]}  \geq \frac{1}{4 \alpha} \frac {\norm{\dz}^2}{\norm{\left[ \begin{matrix} \dl \\ \dn \end{matrix} \right]}}, 
\end{equation*} 
and in turn, by \Cref{ass:7}, there exist a constant $M$ such that:
\begin{equation*}
\frac{\norm{\left[ \begin{matrix} g+s \\ h \end{matrix} \right]}}{\norm{\left[ \begin{matrix} \dl \\ \dn \end{matrix} \right]}}  \geq \frac{1}{4 \alpha} \frac {\norm{\dz}^2}{\norm{\left[ \begin{matrix} \dl \\ \dn \end{matrix} \right]}^2} = \frac{M}{4 \alpha} > 0,
\end{equation*}
for all sufficiently large iterations $i$, hence the constraints are bounded from below in norm. By \Cref{lemma:unbounded_rho}, the penalty parameter $\rho$ must be bounded over the infinite subsequence of iterations, contradicting the unboundedness assumption. 
\end{proof}

\smallskip

\begin{lem}
\label{lemma:t1_1}
Under \Cref{ass:7}, the condition in  \eqref{eq:Wolfe_condition_2} holds with step size $t=1$ for sufficiently large $i$, i.e.:
\begin{equation*}
\phi(1) - \phi(0) \leq \sigma_1 \phi'(0),
\end{equation*}
where $0 < \sigma_1 < \frac12$.
{\hfill $\square$}
\end{lem}

\smallskip

\begin{proof}
The proof is similar to \cite[Lemma 5.2]{gill_SQP_theory} and \cite[Lemma 4.2]{powell1986recursive}, see \Cref{app:proofs} for the proof details. 
\end{proof}

\smallskip
\begin{lem}
\label{lemma:t1_2}
Under \Cref{ass:7}, the condition in  \eqref{eq:Wolfe_condition_1} holds with step size $t=1$ for sufficiently large $i$, i.e.:
\begin{equation*}
|\phi'(1) | \leq \sigma_2 |\phi'(0)|
\end{equation*}
where $\sigma_2 < \frac12$.
{\hfill $\square$}
\end{lem}
\smallskip

\begin{proof}
The proof is similar to \cite[Lemma 5.3]{gill_SQP_theory}. See \Cref{app:proofs} for the proof details.
\end{proof}
\smallskip


\section{Practical implementation}

In \Cref{sec:Proof} we have proven the convergence of the proposed method for general inequality and equality constrained, smooth, optimization problems. On the other hand, it is known that the projected gradient method is inefficient if the feasible set is a general polytope \cite{bertsekas1999nonlinear}.

In the following, we outline two methods for simplifying the computation of the projection. In \Cref{sec:squared_slacks} we transform the general nonlinear problem into an equality constrained problem via squared-slack variables, and compute the projection onto the resulting affine subspace in closed form.

In \Cref{sec:MPC_gradient}, we apply the method to nonlinear model predictive control problems with box constraints on the input variables and terminal quadratic constraints on the state variables. As in standard gradient method for linear MPC, by writing (condensing) the state variables as an explicit function of the input, the equality constraints are directly embedded into the objective function. Then, the considered constraints are shown to be easy-to-project for the proposed algorithm by the introduction of the squared-slack variables.

\subsection{General optimization problems}
\label{sec:squared_slacks}

The problem in \eqref{eq:original_problem} can be reformulated as the equality constrained problem
\begin{equation*} 
\begin{split}
\min_{ (z,y) \in \Real^n \times \Real^m} \hspace{0.3cm} &  J(z)  \\
\text{s.t.} \hspace{0.35cm} 
&g(z) + \frac12 \diag(y) \, y = 0 \\
& h(z) = 0, \\
\end{split}
\end{equation*}
%
%
hence more generally as
\begin{equation} 
\label{eq:equality_problem}
\begin{split}
\min_{v \in \Real^{n+m}} \hspace{0.3cm} &  J(v)  \\
\text{s.t.} \hspace{0.35cm} & p(v) = 0,
\end{split}
\end{equation}
with primal variable $v := [z;y]$ and $p:\Real^{n+m} \rightarrow \Real^{m+p}$ defined as $p([z;y]) := [ g(z) +\frac12 \textup{diag}(y)\, y; h(z)]$.
Let us define $\mu := [\lambda; \nu]$ as the dual variables associated to the constraints $p(v)=0$ ($\mu$), $g(z)+ \frac12 \diag(y)y = 0$ ($\lambda$), and $h(z) = 0$ ($\nu$), respectively. 
The equivalence between \eqref{eq:original_problem} and \eqref{eq:equality_problem} is shown in technical details in \Cref{app:equivalence_slacks}.

In \Cref{alg:nonlinear_algorithm}, the primal update $\dv$ is computed in \eqref{eq:algorithm_dz} as a function of the current $v$, i.e.,
\begin{equation*}
\begin{split}
\dv := \Proja{\left\{\dv \in \Real^{n+m} \ |  \ \nabla p(v)^\top \dv = - p(v)\right\} }{ -\alpha \nabla J(v)},
\end{split}
\end{equation*}
where 
$$
p(v) = \left[ \begin{matrix}
g(z) + \frac12 \diag(y) \, y \\ h(z)
\end{matrix} \right], 
\qquad
\nabla p (v) = \left[ 
\begin{matrix}
\nabla g(z) & \nabla h(z) \\ \diag(y) & 0
\end{matrix}   \right] . 
$$
The projection admits a closed form solution. In fact, we can determine the dual variable $\mu\G := \left[ 
 \lambda\G ;  \nu\G
\right] $ as the solution of the dual problem \cite{Boyd2004_ConvexOpt}:
\begin{multline}
\label{eq:eq_dual}
\mu\G := \left[ 
\begin{matrix}
\lambda\G \\ \nu\G
\end{matrix} 
\right]
=
 \left(\alpha \nabla p (v) ^\top \nabla p (v) \right) ^{-1} \left(p(v) -\nabla p(v)^\top \alpha \nabla J(v)  \right)  \\
 = \! \! \left[ \begin{matrix}
\nabla g(z)^\top \nabla g(z) \! + \! \diag (y)^2 &  \!\!\! \nabla g(z)^\top \nabla h(z) \\ \nabla h(z)^\top \nabla g (z) & \!\!\! \nabla h(z)^\top \nabla h(z) 
\end{matrix} \right]^{-1} \cdot \\
\cdot \left[
\begin{matrix}
\frac{1}{\alpha} g(z) \! + \! \frac{1}{2 \alpha} \textup{diag} (y) y \! - \! \nabla g(z)^\top \nabla J (z) \\
\frac{1}{\alpha} h(z) - \nabla h(z)^\top \nabla J(z)
\end{matrix}
 \right].
\end{multline}
Then, the primal solution is given by
\begin{equation}
\begin{split}
\label{eq:eq_primal}
\dv &= \left[ \begin{matrix}
\dz \\ \dy 
\end{matrix} \right] = -\alpha \nabla J(v) - \alpha \nabla p (v)  \mu\G = \\ &= \left[ \begin{matrix} -\alpha \nabla J (z) \\ 0 \end{matrix} \right] - \alpha \left[
\begin{matrix}
\nabla g(z) & \nabla h(z) \\ \diag(y) & 0 
\end{matrix} \right] \left[ \begin{matrix}
\lambda\G \\ \nu\G 
\end{matrix} \right] \\ 
&= \left[ \begin{matrix}
-\alpha \nabla J (z) - \alpha \nabla g(z) \lambda\G - \alpha \nabla h(z)  \nu\G \\ - \alpha \diag(y) \lambda\G
\end{matrix} \right],
\end{split}
\end{equation}
and dual increments $\dm$ from \eqref{eq:dual_alg} are
$
\dm := \left[ \begin{matrix}
\dl \\ \dn 
\end{matrix} \right] = \mu\G - \mu
$.

By \Cref{ass:6}, the matrix  $\nabla p (v)$ can be proven to be full rank, therefore the matrix 
$\left(\alpha \nabla p (v) ^\top \nabla p (v) \right)$  is invertible.
Also note that by \eqref{eq:eq_dual}, the squared-slack variables do not increase the dimension of the matrix, thus the complexity of the matrix inversion does not increase.


The matrix inversion in \eqref{eq:eq_dual}  is commonly obtained via Cholesky factorization, which for a general dense matrix has a worst case computational complexity of \ \ \ \ \ \ $\mc{O}\left((m+p)^3\right)$. Therefore, for the factorization to be efficient, the sparsity of the gradients $\nabla g$ and $\nabla h$ should be exploited. In the following section, we consider a general nonlinear MPC problem for which the matrix inversion can be computed symbolically, thus avoiding the Cholesky factorization.

\subsection{Model predictive control problems}
\label{sec:MPC_gradient}

Sparsity patterns naturally arise in model predictive control problems, due to the causality of the dynamics and the structure of the constraints.

Let us consider a typical nonlinear MPC problem with box input constraints and a quadratic terminal state constraint,
\begin{equation}
\label{eq:original_MPC}
 \begin{split}
\min_{ ( x_{k+1}, u_{k} )_{k=0}^{N-1}} \hspace{0.0cm} &\sum\limits_{k=0}^{N-1}  \left\{  \frac12 x_k^\top Q  x_k + \frac12 u_k^\top R u_k \right\} + \frac12 x_N^\top P x_N \\
\text{s.t.} \hspace{0.6cm} &x_{k+1} = f \left(x_k,u_k\right) \ \forall k \in \mathbb{Z}[0,N-1] \\
& u_k \in \left[a_k, b_k\right] \hspace{0.941cm} 
\forall k \in \mathbb{Z}[0,N-1] \\
&\textstyle \frac12 x_N^\top P x_N \leq c, 
 \end{split}
\end{equation}
where the index $k$ spans the predicted state $x_{k+1}$ and input $u_k$ in the horizon $N$. The bounds satisfy $a_k < b_k \in \mathbb{R}^{n_\text{u}}$ componentwise, $c>0$ and
$Q, P, R \succcurlyeq 0$. The discrete-time dynamics, $f$, are nonlinear, hence the program in \eqref{eq:original_MPC} is in general nonconvex.



 \label{sec:mod_alg}
For ease of notation, let us first define the vectors $\bu := [ u_{0}; \ldots; u_{N-1}]$ for the control input sequence, with bounds $\ba~:=~[ a_{0}; \ldots; a_{N-1} ]$ and $\bb~:=~[ b_{0}; \ldots; b_{N-1} ]$, and the corresponding state evolution 
$\bx~:=~[ x_{1}; \ldots; x_{N} ]$, 
and stack the state and input cost matrices
$\Q = \text{blockdiag} \left(Q,  \ldots, Q, P \right)$ and $\mc{R}~=~\text{blockdiag} \left(R,  \ldots, R \right)$.
We recast the dynamics in a compact form as $\bx = \psi(\bu)$, where for a fixed initial state $x_0$, the function $\psi: \R^{N n_\txu} \rightarrow \R^{N n_\txx}$ maps the sequence of inputs $\bu$ to the predicted  sequence of states $\bx$ according to the nonlinear dynamics $x_{k+1} = f(x_k,u_k)$. Thus, by including the nonlinear dynamics within the objective and by adding the nonlinear slacks $\by\ta, \by\tb \in \Real^{N n_\text{u}}$ and $y\tc \in \Real$ as in \eqref{eq:equality_problem}, the
 MPC problem in \eqref{eq:original_MPC} reads as  
\begin{equation}
\label{eq:nonlinear_MPC}
\begin{split}
	\min_{\bu, \by\ta, \by\tb, y\tc} \hspace{0.4cm} & \textstyle \frac12 \psi \left( \bu \right)^\top \Q  \psi \left( \bu \right) + \frac12 \bu^\top \mc{R} \bu =: J(\bu) \\
	\text{s.t.} \hspace{0.4cm} & -\bu + \bs a + \frac12 \diag (\by\ta) \by\ta = 0 \\
	 & \bu - \bs b + \frac12 \diag (\by\tb) \by\tb = 0  \\
	& \textstyle \frac12 \psi_N (\bu)^\top P \psi_N (\bu) - c + \frac12 y\tc^2  = 0.  \\
\end{split}
\end{equation}
This formulation, albeit being unusual compared to other approaches when solving nonlinear MPC problems \cite{Diehl2002577, Tenny2004}, leads to computational advantages for the proposed \Cref{alg:nonlinear_algorithm}. 

The primal and dual variable updates of \Cref{alg:nonlinear_algorithm} are determined as explained in \Cref{sec:squared_slacks}. 
Note that the matrix inversion in \eqref{eq:eq_dual} can be computed analytically offline. In fact, since the gradient of the constraint is $\nabla g(\bu) = \left[ \ -I \ | \ I  \ | \ q \  \right]$, where $q \in \Real^{N n_\text{u}}$ is the gradient of the terminal constraints with respect to $\bu$, the matrix is inverted as follows:
\begin{equation*}
\begin{split}
&\left(\nabla g (\bu) ^\top \nabla g (\bu) + \diag([\by\ta;\by\tb;y\tc])^2\right)^{-1}   \\
&= \left[ 
\begin{array}{ccc}
I + \diag(\by\ta)^2 & -I & -q \\
-I & I + \diag(\by\tb)^2  & q \\
-q^\top & q^\top & q^\top q + y\tc^2
\end{array}
\right]^{-1} \\
&=
\left[ \begin{array}{ccc} D + B + r (B q) (B q)^\top & D  - r (B q)(A q)^\top & r B q \\ D - r (A q)(Bq)^\top & D + A + r (A q) (A q)^\top & -r A q \\ r (Bq)^\top & -r (A q)^\top & r
\end{array} \right],
\end{split}
\end{equation*}
with
\begin{equation*}
\begin{split}
D &= \diag \left(\frac{1}{y_{\text{a},j}^2 + y_{\text{b},j}^2 + y_{\text{a},j}^2 y_{\text{b},j}^2}\right), \ A = \diag \left(\frac{y_{\text{a},j}^2}{y_{\text{a},j}^2 + y_{\text{b},j}^2 + y_{\text{a},j}^2 y_{\text{b},j}^2}\right) \\
B &= \diag \left(\frac{y_{\text{b},j}^2}{y_{\text{a},j}^2 + y_{\text{b},j}^2 + y_{\text{a},j}^2 y_{\text{b},j}^2}\right) , \ r = \left( \sum_{i=1}^{N n_\text{u}}  \frac{y_{\text{a},j}^2 y_{\text{b},j}^2}{y_{\text{a},j}^2 + y_{\text{b},j}^2 + y_{\text{a},j}^2 y_{\text{b},j}^2} q_j^2 + y\tc^2 \right)^{-1}.
\end{split}
\end{equation*}
Because of the diagonal structure of $A$ and $B$, the vectors $Aq$ and $Bq$ are cheap to compute, and this allows one to compute the matrix multiplication in \eqref{eq:eq_dual} in only $\mc{O}\left(N n_\txu\right)$ floating point operations (FLOPS).
Moreover, since the terminal constraint in \eqref{eq:original_MPC} has the same structure of the terminal cost in the objective function, the computation of $q$ is inexpensive when performed together with the computation of $\nabla J(\bu)$.


The primal variable updates $\Du$ and the slack updates $\Dya$, $\Dyb$ and $\Dyc$ 
follow from  \eqref{eq:eq_primal}. The computation of the gradient of the objective function can also be done efficiently by exploiting the causality of the nonlinear dynamics, $f$.  From  the definition of $J(\bu)$ in \eqref{eq:nonlinear_MPC}, we have
\begin{subequations}
\label{eq:du_algorithm_MPC}
\begin{align}
\nabla J(\bu) &= \nabla \psi(\bu) \mc{Q} \bx + \mc{R} \bu \label{eq:du_algorithm_MPC_2} \ \text{ with }  \ \bx = \psi(\bu), 
\end{align}
\end{subequations}
where the matrix $\nabla \psi (\bu)$ contains the standard linearization matrices of the nonlinear dynamics 
%
$$
\textstyle F_{k} := \frac{\partial f}{\partial x}( x_k, u_k), \qquad 
G_{k} := \frac{\partial f}{\partial u}( x_k, u_k),
$$
and it is block upper triangular. That is, $\nabla \psi = [\nabla \psi_0, \ldots,\nabla \psi_{N-1} ]$, where the column blocks $\nabla \psi_j : \R^{N n_\txu} \rightarrow \R^{N n_\txu \times n_\txx}$ for all $j \in \{0,\ldots,N-1 \}$ are as follows,
\begin{equation*}
\begin{split}
\nabla \psi_0 &= \left[ 
\begin{array}{c|c}
 G_0 & 0 
\end{array}
\right]^\top \\
\nabla \psi_1 &= \left[ 
\begin{array}{cc|c}
 F_1 \nabla x_1 (u_0) & G_1 & 0
\end{array} \right]^\top = \left[ 
\begin{array}{cc|c}
 F_1 G_0 & G_1 & 0
\end{array}
\right]^\top \\
 & \ \ \vdots \\
\nabla \psi_{N-1} &= \left[ 
\begin{array}{ccc}
 F_{N-1} \nabla x_{N-1} (u_0, \ldots, u_{N-2})  &  F_{N-1} F_{N-2} \nabla x_{N-2} (u_0, \ldots, u_{N-3})  &  \ldots  \end{array} \right. \\
 & \left. 
\begin{array}{ccc}
 &  \ldots &  G_{N-1} 
\end{array} \right]^\top \\
&= \left[ 
\begin{array}{ccccc}
 F_{N-1} \cdot \ldots \cdot F_1 G_0 & F_{N-1} \cdot \ldots \cdot F_2 G_1 & \ldots & F_{N-1} G_{N-2} & G_{N-1} 
\end{array} \right]^\top.
\end{split}
\end{equation*}
Note that the matrix $\nabla \psi$ is not required explicitly in \eqref{eq:du_algorithm_MPC}, but only the   product  $ \nabla \psi \mc{Q} \bx$ is. Therefore, the block upper triangular structure of $\nabla \psi$ and the block diagonal structure of $\mc{Q}$ yields,
\begin{equation*}
\begin{split}
\nabla \psi \, \mc{Q} \bx &= \left[ 
\begin{matrix}
G_0^\top  & \ldots & & & \\
&  \ddots & & & \\
&  & G_{N-3}^\top &  G_{N-3}^\top F_{N-2}^\top & G_{N-3}^\top F_{N-2}^\top F_{N-1}^\top \\
&  & & G_{N-2}^\top & G_{N-2}^\top F_{N-1}^\top \\
& & & & G_{N-1}^\top
\end{matrix}
\right] 
\left[
\begin{matrix}
Q x_1 \\
\vdots \\
Q x_{N-2} \\
Q x_{N-1} \\
P x_{N} 
\end{matrix}
\right] \\
&= \left[ 
\begin{matrix}
\vdots \\ G_{N-3}^\top \left( Q x_{N-2} + F_{N-2}^\top \left( Q x_{N-1} + F_{N-1}^\top P x_N \right) \right)  \\
 G_{N-2}^\top \left( Q x_{N-1} + F_{N-1}^\top \left(P x_N\right) \right)  \\ 
 G_{N-1}^\top P x_N
\end{matrix}
\right],
\end{split}
\end{equation*}
which can be efficiently determined by backward substitution. For diagonal cost matrices $Q$, $R$ and full $P$, the number of Floating Point Operations (FLOPS) required to compute the last vector $G_{N-1}^\top P x_N$ is upper bounded by $2 (n_\txx^2 +  n_\txx n_\txu)$ FLOPS. By not recomputing the term $P x_N$, the second last subvector requires a number of FLOPS upper bounded by $2 (n_\txx^2  + n_\txx n_\txu+ n_\txx)$ FLOPS. Since every subvector can be computed from the successive one and by also considering the term $\mc{R} \bu$, the computational complexity of computing the gradient step is $\mc{O} \left( N(n_\txx^2  + n_\txx n_\txu) \right)$.
Since the subsequent steps of \Cref{alg:nonlinear_algorithm} to determine   $\rho\ii$ from \eqref{eq:def_phi} and $t\ii$ from \eqref{eq:Wolfe_condition}  have lower complexity, including the computation of the merit function $\phi(t)$ and its derivative $\phi'(t)$, this is the resulting complexity of the algorithm.  
Note that this complexity is comparable to standard gradient method for linear MPC \cite{richterFastMPC}, $\mc{O}\left((N n_\txu)^2\right)$,  and since the Hessian $H$ never needs to be computed, the complexity depends linearly (instead of quadratically) on the prediction horizon, $N$. Further, the complexity of the SQP method depends on the complexity of the QP solver. The  Active Set Method for the presented MPC problem requires $\mc{O}\left((N n_\txu)^2\right)$ FLOPS \cite{ferreau2008activeSet}, while an Interior Point Method exploiting sparsity of the MPC requires  $\mc{O}\left(N(n_\txx^3 + n_\txx^2 n_\txu)\right)$ FLOPS \cite{domahidi2012FORCES}.
 
Note that the Lyapunov constraint $\frac{1}{2} x_N^\top P x_N$ need not  necessarily act on the terminal state: an alternative formulation such as contractive MPC, employing quadratic Lyapunov conditions in time steps other than the last one \cite{Kothare2000}, can be analogously considered.


\section{Numerical example: nonlinear Model Predictive  Control of an inverted pendulum}
\label{sec:simulations}
We consider an inverted pendulum 
as a rod of length $l = 0.3$ m  with mass $m = 0.2$ kg concentrated at the tip and no friction acting on the cart and swing. The mass of the cart is $M = 0.5$ kg and the gravitational acceleration $g = 10$~m/s$^2$. The states $x_1$ and $x_2$ are respectively the cart position and velocity and $x_3$ and $x_4$, the pendulum angle and angular velocity. The input $u$ is the applied force on the cart  and it is subject to box constraints. We discretize the continuous-time dynamics
%
%
\begin{equation*}
\begin{split}
&\dot{x}_1 = x_2 \\
&\dot{x}_2 =  \frac{m g \sin x_3 \cos x_3-m l x_4^2 \sin x_3 +u}{M+m \sin^2(x_3)}\\
&\dot{x}_3 = x_4 \\
&\dot{x}_4 =  \frac{g}{l}\sin  x_3  \, +  \frac{ m g \sin  x_3   \cos^2  x_3  + u \cos  x_3  - m l x_4^2 \sin x_3  \cos x_3 }{l \left(M + m\sin^2  x_3 \right)},
\end{split}
\end{equation*}
%
with the explicit Euler method with sampling time $T_\text{s} = 0.1$~s, and hence obtain an MPC problem of the form \eqref{eq:original_MPC}, with model Jacobians computed symbolically. The desired closed-loop performance is achieved with a prediction horizon of $N=8$ and the cost matrices are $Q = \diag([10; \, 0.1; \, 100; \, 0.1])$ and $R = 1$. The terminal cost matrix, $P$, is determined by the Algebraic Riccati Equation using  linearized dynamics around the desired equilibrium. The constant $c$ in the terminal constraint is set to $c=1.5$, so that such constraint is active at the first iteration.

The system has two sets of unforced equilibria, the unstable ones $[p; \, 0; \, 2 k \pi; \, 0]$, and the stable ones $x = [  p; \, 0; \, \pi + 2 k \pi; \, 0 ]$, with $p \in \Real$ and $k \in \mathbb{Z}$. Physically, the former correspond to the pendulum in the upright position, while the latter in the natural upside-down configuration. The goal of the controller is to stabilize the system around the origin, that is, to the unstable equilibrium, starting from the stable one at	 $x_0 := \left[ 0; \, 0; \, \pi; \, 0 \right]$. 



For many nonlinear problems, the Real-Time Iteration yields a sufficiently good approximation of the nonlinear solution \cite{Diehl2002577}. This consists of solving only one QP in \eqref{eq:QP_z} at every time step. For the specific problem considered here, the RTI effectively stabilizes the pendulum, but its closed-loop cost is much larger than that obtained by a full nonlinear solution, as shown in \Cref{fig:statesAll} and \Cref{tab:Computations}, and it has the advantage of requiring low computational times \cite{FORCESPro}. Further, the linearization of the terminal constraint may result to be unfeasible, even though the terminal constraint is feasible for the original nonlinear problem. Thus, a reformulation with soft constraints is required.

The nonlinear solution is determined via the proposed method and with the Sequential Quadratic Programming approach in \Cref{alg:SQP}, and the computational times obtained are in \Cref{tab:Computations}. Specifically, the solver SNOPT is used, running in Fortran and interfaced via TOMLAB to {\sc  Matlab}, and the computational time is measured internally by the solver \cite{gill2005snopt, holmstrom1999tomlab}. We do not observe a significant difference in the computational time with direct calls to SNOPT.  The proposed algorithm has been coded in C and the time indicated comprises both the preparation of the problem, i.e. building the Jacobian matrices, and solving the optimization problem. The calculations have been performed on a commercial off-the-shelf Windows PC with processor Intel Core i7-3740QM $2.70$Ghz. A pure sequential C code was used, compiled by Intel Composer 2016 with the optimization flag \texttt{/Ox} enabled.  

In the implementation we have noticed that the solver SNOPT requires an overly long time to solve the first optimization problem. As the level of optimality does not reach the desired tolerance, the solver exceeded the maximum number of major iterations (150), thus requiring more than $1$ second for the solution of the problem. The average and worst case times reported in \Cref{tab:Computations} for the SQP solver do not account for this first iteration. The problem does not occur in the  proposed algorithm with the same initialization. 

The computational times obtained make the proposed algorithm competitive with the solver SNOPT. The average time is $80 \%$ faster than the SQP, while the best and worst cases are significantly better. Warm starting makes the algorithm particularly effective when it is initialized close to the optimal solution, as only few gradient steps are required for convergence. The proposed algorithm is slower than the RTI by around one order of magnitude, as the RTI requires only the solution to one QP at each time step. 


Additional benefits in terms of computational speed can be obtained by accelerating the proposed algorithm via a specific heuristic that modifies \eqref{eq:algorithm_dz}. More details and computational times for a similar example are given in \cite{Alg_preliminary_CDC2016}.

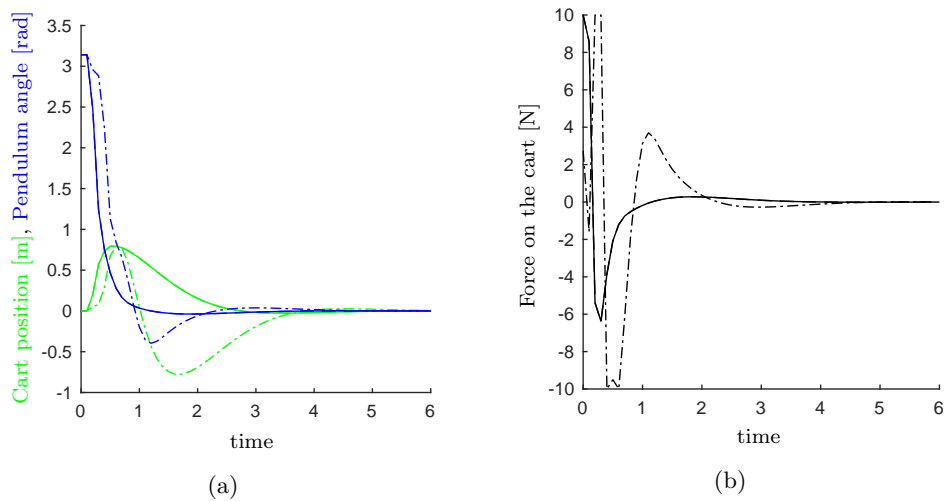
\begin{figure}[t]
	\centering
	\begin{subfigure}[c]{0.45\columnwidth}
	\begin{tikzpicture}[scale=1]
		\node[anchor=south west,inner sep=0] at (0,0) {\includegraphics[width=.9\columnwidth]{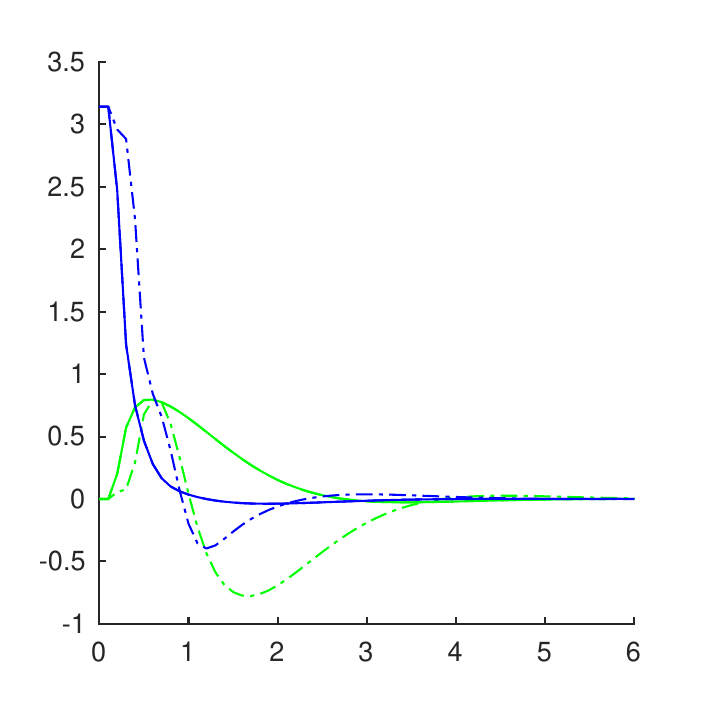}};
		\node at (-.25,2.8) {\rotatebox{90}{  {\footnotesize {\color{green} Cart position [m]}, {\color{blue} Pendulum angle [rad]}}}};
		\node at (2.8,-0.25) {\footnotesize time};
    \end{tikzpicture}	
    \caption{ }
    \end{subfigure}
    \hspace{0.05\columnwidth}
    \begin{subfigure}[c]{0.45\columnwidth}
	\begin{tikzpicture}[scale=1]
		\node[anchor=south west,inner sep=0] at (0,0) {\includegraphics[width=.9\columnwidth]{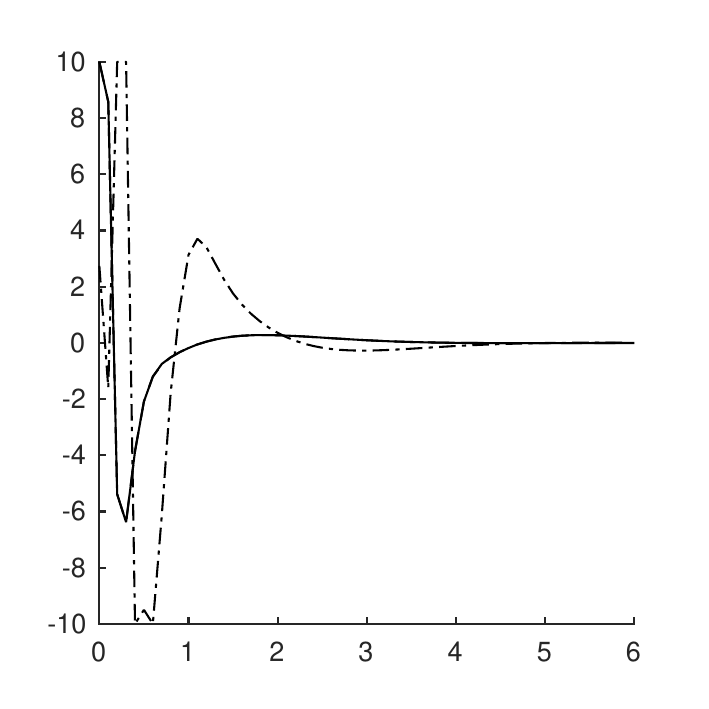}};
		\node at (-.25,2.8) {\rotatebox{90}{  {\footnotesize Force on the cart [N] }}};
		\node at (2.8,-0.25) {\footnotesize time};
    \end{tikzpicture}	
        \caption{ }
    \end{subfigure}
    \caption{Swing-up simulation of an inverted pendulum: (a) state dynamics, where the angle is in blue and the cart position in light green; (b) input: force applied on the cart. The solid lines show the proposed gradient algorithm solution and they overlap with those of the Sequential Quadratic Programming solution. The dash-dot lines show the Real-Time Iteration solution. }
    \label{fig:statesAll}    
\end{figure}

\begin{table}
\caption{Computational times}
\label{tab:Computations}
\begin{center}
\begin{tabular}{@{}lllll@{}}\toprule
 Method  & Avg. time (ms) & Best (ms) & Worst (ms) &  Cost   \\ 
\midrule
Proposed Algorithm & $2.39$ &  $0.06$ & $4.15$ & $318$ \\
 SQP (SNOPT) & $9.66$* & $4.00$ &  
 $370$* &  $318$ \\
 RTI (FORCES PRO) & $0.13$ & $0.10$ & $0.17$ & $527$ \\
\bottomrule
\multicolumn{5}{p{12cm}}{* The first MPC optimization in SNOPT exceeds the maximum number of major iterations and is not considered in the average and worst case times.}
\end{tabular}
\end{center}
\end{table} 

\begin{figure}[t]
	\centering
	\begin{tikzpicture}[scale=1]
		\node[anchor=south west,inner sep=0] at (0,0) {\includegraphics[width=0.45\columnwidth]{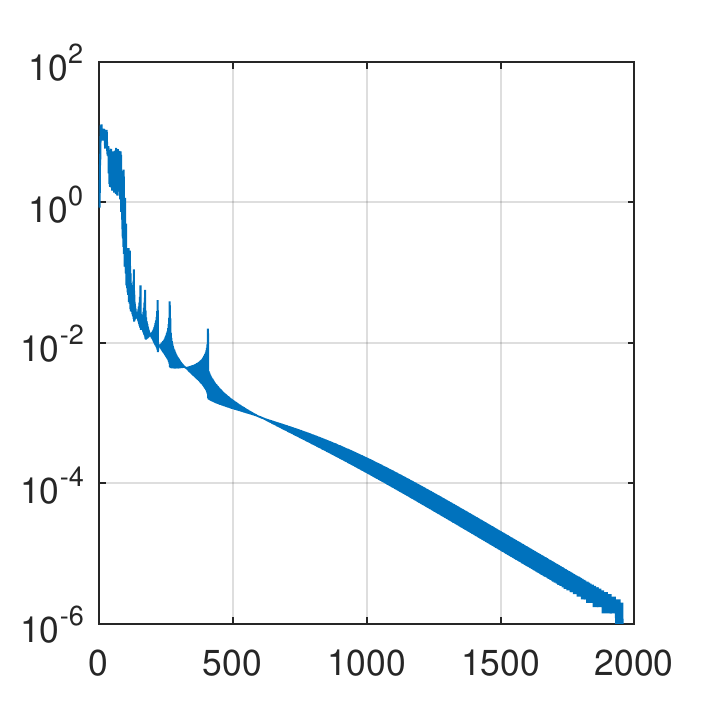}};
		\node at (-0.3,3) {\rotatebox{90}{  {\footnotesize KKT optimality}}};
		\node at (3.2,-0.2) {\footnotesize \# iterations};
    \end{tikzpicture}	
    \caption{Level of optimality (derivative of the Lagrangian) as a function of the iteration number for a particular MPC instance. The desired tolerance is set to $1$e$-6$.}
    \label{fig:converg}    
\end{figure}

\subsection{Computational considerations}

We consider standard convergence conditions to terminate the algorithm. Instead of considering an exact null $\dz\ii$ as in \Cref{alg:nonlinear_algorithm} for terminating the optimization routine, we relax the condition to a small enough value. Leveraging \eqref{eq:dz_KKT}, if feasibility of the constraint is achieved, by bounding appropriately $\norm{\dz}$ we can equivalently impose that the iterate $z\ii$ is a KKT point for the nonlinear problem with a desired tolerance. Another condition involves the derivative of the augmented Lagrangian function, $ \phi'(0) $. 
By  \eqref{eq:condition_phi_PM}, setting a negative lower bound to the derivative $\phi'(0)$ corresponds to checking that the iterate $\norm{\dz}$ is small. As usual in gradient methods, a limit on the maximum number of iterations is also set. We discuss this last criterion in the considerations listed next.

The theoretical results presented in \Cref{sec:Proof} show that  guarantees on the convergence speed can be derived only when the iterate is close to the solution and the correct active set is determined, achieving linear rate. In fact, by \Cref{th:local_conv}, linear convergence is achieved via a specific tuning of 
$\alpha$, based the maximum eigenvalue of Hessian of the Lagrangian at the optimum. This would allow for $t=1$ as shown in \Cref{sec:t_1}. On the other hand, such a value is unknown a priori, and an approximation based on the Hessian of the current iterate can be overly expensive to compute. From our numerical experience, we  recommend instead to set an $\alpha$ (not necessarily linked to the problem Lagrangian) and let the line search variable $t$ become smaller than $1$ in the line search. In fact, this approach achieves a comparable convergence speed to the linear rate obtained with the specific $\alpha$ of \Cref{sec:t_1} and $t=1$. 


From our numerical experience we also observe that the convergence rate might be faster than linear when the solution is far from optimality, albeit theoretically the convergence rate is unknown, and then slow down to linear once it is close to the optimal point. \Cref{fig:converg} shows, as an example, the KKT optimality, computed via the iterate norm $\dz$, as a function of the iteration number for a particular MPC instance. Analogously to standard convex optimization theory, the linear rate factor depends on second-order information and, in our case, a large condition number of the  Hessian of the Lagrangian at the solution can yield a  slow linear convergence in practice \cite{bertsekas1999nonlinear}.
%
%
This means that most of the iterates might be employed to reach the desired tolerance.  In the particular example considered here we set as a limit $3000$ iterations.

Combining the method with standard SQP would be a convenient solution for a general purpose solver, since using second-order information as the correct active set is determined speeds up convergence to superlinear.

\appendix

\section{Proofs of \Cref{sec:Proof}}
\label{app:proofs}

\begin{proof}[Proof of \Cref{th:primal_dual}]
It follows from \Cref{th:convergence} and \Cref{lemma:active_set} that if \\ 
$\norm{\dz\ii} = 0$, then, by \eqref{eq:FOC}, \eqref{eq:update_dual} and \eqref{eq:dz_KKT},  $\lambda\G\ii= \lambda^\star$   and $\nu\G\ii = \nu^\star$, and the algorithm terminates. Thus, we will assume henceforth that $\norm{\dz\ii} \neq 0$ for all $i \in \N$.
By defining, for all $k \leq i \in \N$,
\begin{equation*}
\gamma_{k,i} := 
\begin{cases}
\tilde t\ii & \text{ if } k=i \\
\tilde t\kk \prod_{m = k+1}^{i} \left( 1 - \tilde t\mm \right) & \text{ otherwise, }
\end{cases}
\end{equation*}
with $\tilde t^{(0)} := 1$ and $\tilde t^{(k)} = t^{(k)}$,
the definition in \eqref{eq:dual_def_2} implies that 
\begin{equation}
\label{eq:th2_be0}
\lambda\ip = \sum_{k=0}^i \gamma_{ki} \lambda\G\kk,
\end{equation}
 for all $i\geq 0$, because of the initial condition $\lambda^{(0)} = \lambda\G^{(0)}$. Then, by \Cref{lemma:phi_property} we have
\begin{subequations}
\begin{align}
0 < \bar t \leq \tilde{t}\ii &\leq 1 \ \ \forall i \in \N \label{eq:th2_be1}\\
\sum_{k=0}^{i} \gamma_{k,i} &= 1 \label{eq:th2_be2}\\
\gamma_{k,i} & \leq (1-\bar t)^{i-k} \ \ \forall k < i. \label{eq:th2_be3}
\end{align}
\end{subequations}

Since $z\ii \rightarrow z^\star$, the iterates will reach a neighborhood of $z^\star$ where the problem in \eqref{eq:proj_opt} identifies the correct active set (\Cref{lemma:active_set}) and the active constraint will have full column rank. Assume that the property holds for $i > \tilde i$. From \eqref{eq:dz_KKT}, the definition of $\lambda\G\ii$ and \Cref{ass:4}, for $i\geq \tilde i$ there exists $M > 0$ such that:


\begin{equation}
\label{eq:th2_be4}
\lambda\G\ii = \lambda^\star + M\ii d\ii v\ii
\end{equation}
with $|M\ii| \leq M$, $d\ii = \max\left\{\norm{\dz\ii},\norm{z^\star - z\ii}\right\}$ and $\norm{v\ii} = 1$. 
{For any given} $\varepsilon > 0$, \Cref{th:convergence} {implies} that $i_1$ can be chosen such that for all $i \geq i_1$, {we have}
\begin{equation}
\label{eq:th2_bo}
|M\ii d\ii| \leq \varepsilon/2.
\end{equation}
Then, we define the iteration index $i_2$ such that for all $i \geq i_2$, {we have}
\begin{equation}
\label{eq:th2_be5}
(1 - \bar t)^i \leq \frac{\varepsilon}{2 (i+1) (1 + \hat \lambda\G + \norm{\lambda^\star})}
\end{equation}
with $\hat \lambda\G$ {being} an upper bound {to} $\norm{\lambda\G\ii}$ for all $i$. {Now} let $\tilde i := \max \{i_1, i_2\}$. Then, from \eqref{eq:th2_be0} and \eqref{eq:th2_be4}, for all $i \geq \tilde i$ we have 
$$
\lambda\ip = \sum_{k=0}^{\tilde i} \gamma_{k,i} \lambda\G\kk + \sum_{k=\tilde i + 1}^{i} \gamma_{k,i} 
\left(\lambda^\star + M\kk d\kk v\kk \right).
$$
Hence it follows from \eqref{eq:th2_be2} that
\begin{equation*}
\lambda\ip - \lambda^\star = \sum_{k=0}^{\tilde i} \gamma_{ki}( \lambda\G\kk - \lambda^\star) + \sum_{k=\tilde i + 1}^{i} \gamma_{ki} M\kk d\kk v\kk.
\end{equation*}
Since the dual variable $\lambda\G$ and $v\kk$ are bounded in norm, it follows that
\begin{equation}
\label{eq:th2_int1}
\norm{\lambda\ip - \lambda^\star} \leq ( \hat \lambda\G + \norm{\lambda^\star}) \sum_{k=0}^{\tilde i} \gamma_{k,i} + \sum_{k=\tilde i + 1}^{i} \gamma_{k,i} | M\kk d\kk |.
\end{equation}
For all iterations $i \geq 2 \tilde i$, it follows from \eqref{eq:th2_be1} and \eqref{eq:th2_be3} that
\begin{equation*}
\sum_{k=0}^{\tilde i} \gamma_{k,i} \leq \sum_{k=0}^{\tilde i} (1 - \bar t)^{i-k} \leq \sum_{k=0}^{\tilde i} (1 - \bar t)^{2 \tilde i - k} \leq (\tilde i + 1) (1 - \bar t)^{\tilde i}.
\end{equation*}
{By} using \eqref{eq:th2_be5}, we derive the bound 
{
$
\left( \hat \lambda\G + \norm{\lambda^\star} \right) \sum_{k=0}^{\tilde i} \gamma_{k,i} \leq \frac12 \varepsilon.
$
on the first term on the right-hand side of \eqref{eq:th2_int1},
and to bound the second term in \eqref{eq:th2_int1}, we use \eqref{eq:th2_be2} and \eqref{eq:th2_bo}:
}
\begin{equation}
\label{eq:th2_int2}
\sum_{j=\tilde i + 1}^{i} \gamma_{k,i} |M\kk d\kk| \leq \frac12 \varepsilon \sum_{j=\tilde i + 1}^{i} \gamma_{k,i} \leq \frac12 \varepsilon.
\end{equation}
{Finally, by} combining \eqref{eq:th2_int1} and \eqref{eq:th2_int2}, we obtain that {for all} $\varepsilon$, {there exists} $\tilde{i}$ such that
\begin{equation*}
\norm{\lambda\ii - \lambda^\star} \leq \varepsilon \text{ for all } i \geq 2 \tilde i + 1,
\end{equation*}
which implies that
$
\lim_{i \rightarrow \infty} \norm{\lambda\ii - \lambda^\star} = 0.
$
The convergence of the dual variables for the equality constraint is analogous.
\end{proof}
 
\begin{proof}[Proof of \Cref{lemma:t1_1}]
By continuity of the second derivative we can derive the following relation between the objective function evaluations:
 \begin{equation}
 \begin{split}
 J (z+\dz) &= J(z)  + \frac12 \nabla J (z)^\top \dz + \frac12 \nabla J (z)^\top \dz + \frac12 \dz^\top \nabla^2 J(z) \dz + o\left(\norm{\dz}^2\right) \\
 &= J(z)  + \frac12 \nabla J (z)^\top \dz + \frac12 \left(\nabla J(z+\dz) - \nabla^2 J(z) \dz  \right)^\top \! \dz  + \frac12 \dz^\top \nabla^2 J(z) \dz\\ &\qquad + o \left(\norm{\dz}^2\right) \\
 &= J(z) + \frac12 (\nabla J (z) + \nabla J(z + \dz ))^\top \dz + o \left(\norm{\dz}^2\right) \label{eq:4.9-J} \\ 
\end{split}
 \end{equation}
 and analogously for the constraint functions $g$ and $h$:
\begin{align}
 g(z+\dz) &= g(z) + \frac12 (\nabla g (z) + \nabla g(z + \dz ))^\top \dz + o\left(\norm{\dz}^2\right), \label{eq:4.9-g} \\
 h(z+\dz) &= h(z) + \frac12 (\nabla h (z) + \nabla h(z + \dz ))^\top \dz + o\left(\norm{\dz}^2\right). \label{eq:4.9-h}
 \end{align} 
By \Cref{ass:7} we have that
\begin{equation}
\label{eq:cont_nabla}
\begin{split}
\nabla J(z+\dz) &= \nabla J(z^\star) + \nabla^2 J(z^\star) \left(z+ \dz  - z^\star \right) + o\left(\norm{z+ \dz- z^\star}\right) \\
&= \nabla J(z^\star) +  o\left(\norm{\dz}\right)\\
\nabla g(z+\dz) &= \nabla g(z^\star) + o\left(\norm{\dz}\right) \\
\nabla h(z+\dz) &= \nabla h(z^\star) + o\left(\norm{\dz}\right),
\end{split}
\end{equation}
hence substituting into \eqref{eq:4.9-J}, \eqref{eq:4.9-g}, \eqref{eq:4.9-h} we obtain:
\begin{equation}
\label{eq:cont_fun}
 \begin{split}
 J(z+\dz) &= J(z) + \frac12 (\nabla J (z) + \nabla J(z^\star ))^\top \dz + o\left(\norm{\dz}^2\right), \\
 g(z+\dz) &= g(z) + \frac12 (\nabla g (z) + \nabla g(z^\star ))^\top \dz + o\left(\norm{\dz}^2\right), \\
 h(z+\dz) &= h(z) + \frac12 (\nabla h (z) + \nabla h(z^\star ))^\top \dz + o\left(\norm{\dz}^2\right).
 \end{split}
 \end{equation}
Let us use the simplified notation and denote $J(z)$ as $J$, $J(z+\dz)$ as $J^+$ and $J(z^\star)$ as $J^\star$ (and similarly for $g$ and $h$ and their gradients).
By \eqref{eq:update_dual}, \eqref{eq:ds_def} and \eqref{eq:dual_alg}, $\phi(0)$ and $\phi(1)$ are	
\begin{equation*}
\begin{split}
\phi(0) &= J + \lambda^\top(g+s)+ \nu^\top h +\frac12 \rho (g+s)^\top  (g+s) + \frac12 \rho h^\top  h \,, \\
\phi(1) &= J^+ + \lambda\G^\top \left(g^+-g-\nabla g ^\top \dz \right)+ \nu\G^\top h^+\\ & \qquad +\frac12 \rho \left(g^+ -g-\nabla g ^\top \dz \right)^\top \left(g^+ -g-\nabla g^\top \dz \right)+ \frac12 \rho h^{+ \top}  h^+ \\
&= J^+ + \lambda\G^\top \left(g^+-g-\nabla g ^\top \dz \right)+ \nu\G^\top h^+ + o \left( \norm{\dz}^2 \right),
\end{split}
\end{equation*}
where the last step follows by 
 \eqref{eq:algorithm_dz} and Taylor's expansions
\begin{equation*}
\begin{split}
g^{+} &- g - \nabla g^\top \dz = o\left(\norm{\dz}\right)\\
h^{+} &= \underbrace{h + \nabla h^\top \dz}_{= 0 \, \eqref{eq:algorithm_dz}} + o\left(\norm{\dz}\right) = o\left(\norm{\dz}\right).
\end{split}
\end{equation*}
By \eqref{eq:cont_fun}, $\phi(1)$ can be written as
\begin{equation*}
\begin{split}
\phi(1) &= J + \frac12 (\nabla J + \nabla J^\star)^\top \dz + \frac12 \lambda\G^\top \left( \nabla g^\star-\nabla g\right)\dz \\& \qquad +  \nu\G^\top \left( h + \frac12 (\nabla h + \nabla h^\star))^\top \dz \right)    + o\left(\norm{\dz}^2\right).
\end{split}
\end{equation*}
This implies
\begin{equation}
\label{eq:int_lem}
\begin{split}
\phi(1)-\phi(0) &=  \frac12 (\nabla J + \nabla J^\star)^\top \dz + \frac12 \lambda\G^\top \left( \nabla g^\star-\nabla g\right)^\top \dz - \lambda^\top (g+s) -\nu^\top h \\ & \ +  \nu\G^\top \left( h + \frac12 (\nabla h + \nabla h^\star)^\top \dz \right) -\frac12 \rho \left\| g+s \right\|_2^2 - \frac12 \rho \left\| h \right\|_2^2  
+ o\left(\norm{\dz}^2\right).
\end{split}
\end{equation}
By \Cref{lemma:tune_eta}, the derivative $\phi'(0)$ can be recast as
\begin{equation*}
\begin{split}
\phi'(0) &= \dz^\top \nabla J 
- 2( g +s )^\top \lambda - ( \nabla g^\top \dz + \ds )^\top \lambda\G - \rho \left\| g+s \right\|_2^2 - 2 h^\top \nu + h^\top \nu\G \\ 
&\qquad   - \rho \left\| h \right\|_2^2,
\end{split}
\end{equation*}
which when substituted in \eqref{eq:int_lem} yields
\begin{equation*}
\begin{split}
\phi(1)-\phi(0) &= \frac12 \phi'(0) + \frac12 \left( \nabla J^{\star}+ \nabla g^{\star} \lambda\G  + \nabla h^{\star} \nu\G \right)^\top \dz +\frac12 \lambda\G^\top \ds  + o\left(\norm{\dz}^2\right).
\end{split}
\end{equation*}
Next we note that \Cref{ass:7} implies specific conditions on the convergence of the variables $\lambda\G$ and $\nu\G$. That is, 
\begin{equation}
\label{eq:cond_G}
\begin{split}
\lambda\G - \lambda^\star = \dl + \lambda - \lambda^\star &= o\left(\norm{\lambda - \lambda^\star}\right) = o\left(\norm{\dl}\right) = o\left(\norm{\dz}\right)\\
\nu\G - \nu^\star &= o\left(\norm{\dz}\right).
\end{split}
\end{equation}
Therefore, by using the KKT conditions in \eqref{eq:FOC} we have
\begin{equation*}
\begin{split}
&\nabla J(z^\star)+ \nabla g(z^\star) \lambda\G + \nabla h(z^\star) \nu\G = \nabla J(z^\star)+ \nabla g(z^\star) \lambda^\star + \nabla h(z^\star) \nu^\star \\ &\qquad + \nabla g(z^\star) \left( \lambda\G - \lambda^\star\right) +\nabla h(z^\star)\left( \nu\G - \nu^\star\right) =  o\left(\norm{\dz}\right).
\end{split}
\end{equation*}
Since $\lambda\G^\top \ds \leq 0$ (where equality holds when the correct active set is identified, see the proof of \Cref{lemma:t1_2}), then
\begin{equation*}
\begin{split}
\phi(1)-\phi(0) &\leq \frac12 \phi'(0) +\frac12 \lambda\G^\top \ds + o\left(\norm{\dz}^2\right) \leq \frac12 \phi'(0)  + o\left(\norm{\dz}^2 \right). 
\end{split}
\end{equation*}
To prove that eventually $\phi(1)-\phi(0) - \sigma_1 \phi'(0) \leq 0 $ with $\sigma_1 \in (0,1/2)$, from the previous inequality we have that  $ \phi(1)-\phi(0) - \sigma_1 \phi'(0) \leq \underbrace{\left( \textstyle \frac{1}{2} - \sigma_1 \right)}_{>0} \underbrace{\phi'(0)}_{ < 0} + \, o\left(\norm{\dz}^2 \right)~\leq~0$ for all sufficiently large $i$, since by \Cref{lemma:tune_eta} the term $\phi'(0)$ is upper bounded by a negative term proportional to $\norm{\dz}^2$
%
 and 
 the term $o\left(\norm{\dz}^2 \right)$ vanishes at least quadratically fast as $i$ approaches infinity.
\end{proof}

\begin{proof}[Proof of \Cref{lemma:t1_2}]
Let us use the simplified notation and denote $J(z)$ as $J$, $J(z+\dz)$ as $J^+$ and $J(z^\star)$ as $J^\star$ (and similarly for $g$ and $h$ and their gradients).

By definition, the derivative $\phi'(1)$ is:
\begin{equation*}
\begin{split}
\phi'(1) = \ &\dz^\top \left( \nabla J^+ + \nabla g^+ \,  \lambda\G + \nabla h^+ \, \nu\G + \rho \nabla g^+ \, \left(g^+ -g -\nabla g^\top \dz \right) + \rho \nabla h^+ \, h^+ \right)  \\ 
& + \dl^\top (g^+ -g -\nabla g^\top \dz)+ \dn^\top h^+ + \ds^\top \left( \lambda\G + \rho ( g^+ -g -\nabla g^\top \dz) \right), \\ 
\end{split}
\end{equation*}
and by \eqref{eq:cont_nabla} we have
\begin{equation*}
\begin{split}
\phi'(1) = \ &\dz^\top \left( \nabla J^\star + \nabla g^\star \,  \lambda\G + \nabla h^\star \, \nu\G \right) 
+ \rho \dz^\top \nabla g^+ \, \left(g^+ -g -\nabla g^\top \dz \right)  \\ 
&   + \ds^\top \left( \lambda\G + \rho ( g^+ -g -\nabla g^\top \dz) \right) +  o\left(\norm{\dz}^2\right),
\end{split}
\end{equation*}
where we have used the fourth condition in \Cref{ass:7} and the following relations resulting from 
 \eqref{eq:algorithm_dz} and Taylor expansions: 
\begin{equation*}
g^{+} - g - \nabla g^\top \dz = o\left(\norm{\dz}\right), \ 
h^{+} = \underbrace{h + \nabla h^\top \dz}_{=0} + o\left(\norm{\dz}\right) = o\left(\norm{\dz}\right).
\end{equation*}
Moreover, by \eqref{eq:ds_def} $\ds =  - \nabla g^\top \dz- (g+s)$, hence,
\begin{equation*}
\begin{split}
\phi(1) = \ &\dz^\top \left( \nabla J^\star + \nabla g^\star \,  \lambda\G + \nabla h^\star \,  \nu\G \right) +\ds^\top  \lambda\G+ \\ & \rho (\dz^\top (\nabla g^+ - \nabla g) 
 -(g+s)^\top )  \, \left(g^+ -g -\nabla g^\top \dz \right)      + o \left( \norm{\dz}^2 \right).\\ 
\end{split}
\end{equation*}
%
%
From \eqref{eq:slack_s} and \Cref{ass:7}, $\norm{g+s} = \mc{O} \left( \norm{\dz}\right)$ and ${\ds^\top \lambda\G = 0}$ when the correct active set is determined, see  \Cref{lemma:active_set}. 
In fact, if $g_j<0$ for some $j$ at $z^\star$ then $\lambda\Gj = 0$. Then, if $\rho = 0$, $s_j = -g_j$. Otherwise, if $\rho >0$, as $\lambda_j$ converges to $\lambda^\star_j = 0$, for sufficiently large iterates $i$ we have that  $0\leq \lambda_j < - \rho g_j$. Then, by \eqref{eq:slack_s} we have
$$
s_j = -g_j - \frac{\lambda_j}{\rho} \quad \Rightarrow \quad g_j + s_j = -\frac{\lambda_j}{\rho} = - \frac{\lambda_j - \lambda_j^\star}{\rho} = \mc{O}\left(\norm{\dz}\right).
$$
If $g_j = 0$, then $s_j=0$ and $\lambda\Gj > 0$ by \Cref{ass:6}. By the complementarity conditions in \eqref{eq:dz_KKT} and \eqref{eq:slack_s} it results that $\lambda\Gj ( s_j +  \dsj) = 0$, thus $\dsj = 0$.
Therefore, in any case we conclude that
\begin{equation}
| \rho (\dz^\top (\nabla g^+ - \nabla g)- ( g+s) ) \left(g^+ -g -\nabla g^\top \dz \right)| = o\left(\norm{\dz}^2\right).
\end{equation}
Since by \Cref{ass:7} the quantity $\nabla J^\star + \nabla g^\star \lambda\G  + \nabla h^\star \nu\G$ is $o\left(\norm{\dz}\right)$, then 
$
\phi'(1) = o\left(\norm{\dz}^2\right)
$. 
By \Cref{lemma:tune_eta}, $|\phi'(0)| \geq \frac{1}{2 \alpha} \norm{\dz}^2$, thus \eqref{eq:Wolfe_condition_2} will  eventually be satisfied at every iteration.
\end{proof}

\section{Equivalence of the  squared-slack problem}
\label{app:equivalence_slacks}

The following statement shows the equivalence properties between problems \eqref{eq:original_problem} and \eqref{eq:equality_problem}.

\begin{lem}[{\cite[Proposition 1]{tapia1974stable}, \cite[Section 3.3.2]{bertsekas1999nonlinear}}]
\label{lemma:equivalence_equality}
The following hold:
\begin{enumerate}\setlength\itemsep{0.5em}
\item[(i)] $z^\star$ is a regular solution to  \eqref{eq:original_problem} if and only if $[z^\star;y^\star]$ is a regular solution to \eqref{eq:equality_problem};
\item[(ii)] if $(z^\star,\lambda^\star,\nu^\star)$ is a KKT triple for \eqref{eq:original_problem}, then $([z^\star;y^\star],[\lambda^\star;\nu^\star])$ is a KKT double for \eqref{eq:equality_problem};
\item[(iii)] if $([z^\star;y^\star],[\lambda^\star;\nu^\star])$ is a KKT double for \eqref{eq:equality_problem} and $\lambda^\star \geq 0$, then $(z^\star,\lambda^\star,\nu^\star)$ is a KKT triple for \eqref{eq:original_problem}.
\end{enumerate}
\end{lem}

Note that by the feasibility of the optimal solution to \eqref{eq:equality_problem}, $y^\star$ is such that 
$
y^\star_j = \left( -2 g_j\left(z^\star\right) \right)^{1/2}
$, for all $j \in \{ 1, \ldots, m\}$.

For the proof of part (iii), it is necessary to assume $\lambda^\star \geq 0$; in fact, the first order sufficient conditions for \eqref{eq:equality_problem} may in principle have negative dual variables $\lambda^\star$ associated with the squared-slack equality constraints. Note that the first order conditions, derived with respect to $y$, already guarantee complementarity slackness, i.e., 
$\lambda_j^\star = 0$ for all indexes $j$ of the active inequality constraints at $z^\star$ for \eqref{eq:original_problem}.


The following Lemma shows that the assumption $\lambda^\star \geq 0$ of part (iii) can be dropped if the solution $[z^\star;y^\star]$ is a local minimum.


\begin{lem}
\label{lemma:equivalent_first_order}
If $[z^\star;y^\star]$ is a local minimum of \eqref{eq:equality_problem} with dual variables $[\lambda^\star; \nu^\star]$, then $(z^\star,\lambda^\star,\nu^\star)$ is a KKT triple for \eqref{eq:original_problem}.
\end{lem}

\begin{proof}
By the second order necessary condition for \eqref{eq:equality_problem}, we have 
\begin{equation}
\label{eq:lemma_eq_sec_lagr}
\left[ \begin{matrix}
\tilde z^\top & \tilde y^\top 
\end{matrix} \right] 
\left[ \begin{matrix}
\nabla^2_{zz} \, \Lagr(z^\star,\lambda^\star,\nu^\star) & 0 \\ 0 & \diag(\lambda^\star)
\end{matrix} \right]
\left[ \begin{matrix}
\tilde z \\ \tilde y
\end{matrix} \right] 
\geq \, 0
\end{equation}
for all $\tilde z \in \Real^n$, $\tilde y \in \Real^m$ such that:
\begin{equation}
\label{eq:lemma_eq_sec_cond}
\nabla h(z^\star)^\top \tilde z = 0, \qquad \nabla g_j(z^\star)^\top \tilde z + y^\star_j \tilde y_j = 0 \ \textup{ for all } j \in \{1, \ldots, m\}.
\end{equation}

Let $j$ be the index of an arbitrary active constraint of $z^{\star}$. We can choose $\tilde z = 0$, with $\tilde y_j \neq 0$, and $\tilde y_k = 0$ for all $k\neq j$. Therefore, by \eqref{eq:lemma_eq_sec_lagr}, we obtain $\lambda_j^\star \, \tilde y^2_j \geq 0$, thus $\lambda_j^\star \geq 0$. Then the result follows from  \Cref{lemma:equivalent_first_order}, part~(iii).

\end{proof}

\bibliographystyle{siamplain}
\bibliography{bibl}
\end{document}